\documentclass[10pt]{amsart}

\usepackage[all]{xy}
\usepackage{url}
\usepackage{color}
\usepackage{graphics}
\usepackage[pdftex]{graphicx}

\newcommand{\abs}[1]{\lvert #1\rvert}

\newcommand{\lmd}[1]{#1\text{-}\operatorname{mod}}
\newcommand{\KX}{K_0(\XX)}

\newcommand{\Add}{\operatorname{Add}}
\newcommand{\Hom}{\operatorname{Hom}} 
\newcommand{\Ext}{\operatorname{Ext}} 
\newcommand{\rk}{\operatorname{rk}}   
\newcommand{\slo}{\operatorname{slope}} % same as slope
\newcommand{\lcm}{\operatorname{lcm}}
\newcommand{\ql}{\operatorname{ql}}

\newcommand{\Iff}{\Leftrightarrow}
\newcommand{\df}{\colon}

\newcommand{\del}{\delta}
\newcommand{\lam}{\lambda}
\newcommand{\blam}{{\boldsymbol{\lambda}}}
\newcommand{\Lam}{\Lambda}
\newcommand{\vph}{\varphi}
\newcommand{\Sig}{\Sigma}

\newcommand{\be}{\mathbf{e}}
\newcommand{\bff}{\mathbf{f}}
\newcommand{\bh}{\mathbf{h}}
\newcommand{\bv}{\mathbf{v}}
\newcommand{\bp}{\mathbf{p}}
\newcommand{\base}[1]{\operatorname{B}(#1)}

\newcommand{\CC}{\mathbb{C}}
\newcommand{\scC}{\underline{\mathcal{C}}}
\newcommand{\PP}{\mathbb{P}}
\newcommand{\QQ}{\mathbb{Q}}

\newcommand{\cC}{\mathcal{C}}
\newcommand{\cD}{\mathcal{D}}
\newcommand{\cDb}{\mathcal{D}^b} % bounded derived category
\newcommand{\cO}{\mathcal{O}}
\newcommand{\cT}{\mathcal{T}}

\newcommand{\sD}{\mathsf{D}}
\newcommand{\sE}{\mathsf{E}}

\newcommand{\vc}{\vec{c}}
\newcommand{\vn}{\vec{w}}  % dualizing element
\newcommand{\vx}{\vec{x}}

\newcommand{\oM}{{\bar{M}}}
\newcommand{\tQ}{\widetilde{Q}}
\newcommand{\hQ}{\widehat{Q}}

\newcommand{\dimvec}[6]{\text{\small$\left[
    \begin{matrix}
      #1&#3&#5\\#2&#4&#6
    \end{matrix}\right]$}}
\newcommand{\euler}[1]{\langle #1\rangle}
\newcommand{\transp}{\top}
\newcommand{\ZZ}{\mathbb{Z}}
\newcommand{\NN}{\mathbb{N}}
\newcommand{\XX}{\mathbb{X}}
\newcommand{\QQi}{\mathbb{Q}_\infty}

\newcommand{\unfold}[1]{\mu_{#1}}
\newcommand{\innerfunc}[1]{#1^{(\Inner)}}
\newcommand{\innerfunckl}[1]{\left(#1\right)^{(\Inner)}}
\newcommand{\outerfunc}[1]{#1^{(\Outer)}}
\newcommand{\outerfunckl}[1]{\left(#1\right)^{(\Outer)}}
\newcommand{\iofunc}[2]{#2^{(#1)}}

\newcommand{\typ}[1]{\mathrm{t}(#1)}%{\overline{#1}}%{{{#1}^\bullet}} %{\hat{#1}}
\newcommand{\typc}[1]{\mathrm{t}'(#1)}%{\underline{#1}}%{{{#1}_\bullet}}%{\check{#1}}

\newcommand{\dimv}{\underline{\mathrm{dim}}}
\newcommand{\Exc}{\mathcal{E}}
\newcommand{\End}{\operatorname{End}}
\newcommand{\slope}{\operatorname{slope}}
\newcommand{\coh}{\operatorname{coh}}
\newcommand{\qdist}[1]{\Delta(#1)}

\newcommand{\tAC}{\Delta^{\bowtie}}
\newcommand{\tEG}{\mathbf{E}^{\bowtie}}
\newcommand{\mSet}{\mathcal{B}}
\newcommand{\Inner}{+}%{\mathrm{in}}
\newcommand{\Outer}{-}%{\mathrm{out}}
\newcommand{\complexity}[1]{\gamma(#1)}
\newcommand{\Her}{\mathcal{H}}
\newcommand{\cA}{\mathcal{A}}
\newcommand{\genus}[1]{g_{#1}}

\newcommand{\surf}{\operatorname{\mathbf{S}}}
\newcommand{\marked}{\operatorname{\mathbf{M}}}

\newcommand{\HVCenter}[1]{\setbox 0=\hbox{#1}%
        \dimen0=\wd0%
        \dimen1=\ht0%
        \divide\dimen0 by 2%
        \divide\dimen1 by 2%
        \hskip -\dimen0%
        \lower \dimen1%
        \box0%
        \hskip -\dimen0}
\newcommand{\HBCenter}[1]{\setbox 0=\hbox{#1}%
        \dimen0=\wd0%
        \dimen1=\ht0%
        \divide\dimen0 by 2%
        \hskip -\dimen0%
        \box0%
        \hskip -\dimen0}
\newcommand{\HTCenter}[1]{\setbox 0=\hbox{#1}%
        \dimen0=\wd0%
        \dimen1=\ht0%
        \divide\dimen0 by 2%
        \hskip -\dimen0%
        \lower \dimen1%
        \box0%
        \hskip -\dimen0}
\newcommand{\RVCenter}[1]{\setbox 0=\hbox{#1}%
        \dimen0=\wd0%
        \dimen1=\ht0%
        \divide\dimen1 by 2%
        \hskip -\dimen0%
        \lower \dimen1%
        \box0%
        \hskip -\dimen0}
\newcommand{\RBCenter}[1]{\setbox 0=\hbox{#1}%
        \dimen0=\wd0%
        \dimen1=\ht0%
        \hskip -\dimen0%
        \box0%
        \hskip -\dimen0}
\newcommand{\RTCenter}[1]{\setbox 0=\hbox{#1}%
        \dimen0=\wd0%
        \dimen1=\ht0%
        \hskip -\dimen0%
        \lower \dimen1%
        \box0%
        \hskip -\dimen0}
\newcommand{\LVCenter}[1]{\setbox 0=\hbox{#1}%
        \dimen1=\ht0%
        \divide\dimen1 by 2%
        \lower \dimen1%
        \box0%
        \hskip -\dimen0}
\newcommand{\LTCenter}[1]{\setbox 0=\hbox{#1}%
        \dimen1=\ht0%
        \lower \dimen1%
        \box0%
        \hskip -\dimen0}

\newtheorem{theorem}{Theorem}[section]
\newtheorem{proposition}[theorem]{Proposition}

\newtheorem{corollary}[theorem]{Corollary}
\newtheorem{lemma}[theorem]{Lemma}

\theoremstyle{definition}

\newtheorem{remark}[theorem]{Remark}
\newtheorem{example}[theorem]{Example}
\newtheorem{definition}[theorem]{Definition}

\begin{document}
\title{Tubular cluster algebras I: categorification}
\author{Michael Barot}
\address{\noindent Instituto de Matem\'aticas,  Universidad Nacional Aut\'onoma de M\'exico,
Ciudad Universitaria, 04510 Mexico D.F., Mexico}
\email{barot@matem.unam.mx}

\author{Christof Geiss}
\address{\noindent Instituto de Matem\'aticas,  Universidad Nacional Aut\'onoma de M\'exico,
Ciudad Universitaria, 04510 Mexico D.F., Mexico}
\email{christof@matem.unam.mx}
\date{}

%%%%%%%%%%%%%%%%%%%%%%%%%%%%%%%%%%%%%%%%%%%%%%%%%%
%%%%%%%%%%%%%%%%%%%%%%%%%%%%%%%%%%%%%%%%%%%%%%%%%%

\begin{abstract}
We present a categorification of four mutation finite cluster algebras by 
the cluster category of the category of coherent sheaves over a weighted 
projective line of tubular weight type. Each of these cluster algebras which we call tubular is associated to an elliptic root system.
We show that via a cluster character the cluster variables are in bijection 
with the positive real Schur roots associated to the weighted projective line. 
In one of the four cases this is achieved by the approach to 
cluster algebras of Fomin-Shapiro-Thurston using a 2-sphere with 4 marked points whereas in the remaining cases it is done  by the approach of Geiss-Leclerc-Schr\"oer using
preprojective algebras.
\end{abstract}

\thanks{Both authors thankfully acknowledge support from grant PAPIIT IN103507.}

\maketitle

%%%%%%%%%%%%%%%%%%%%%%%%%%%%%%%%%%%%%%%%%%%%%%%%%%
%%%%%%%%%%%%%%%%%%%%%%%%%%%%%%%%%%%%%%%%%%%%%%%%%%

\section{Introduction}

Cluster algebras were introduced around 2001 by Fomin and 
Zelevinsky~\cite{FZ1} as a tool to study questions concerning dual canonical 
bases and total positivity. Though mainly combinatorial in their conception,
many important questions about cluster algebras were
answered  by introducing some extra structure like for example a 
``categorification''~\cite{FK08},\cite{DWZ09} or some Poisson geometric 
context~\cite{GSV03}.
 
Somehow simplifying, a cluster algebra $\cA(B)$ of rank $n$ with trivial 
coefficients is determined by the  skew symmetrizable exchange matrix 
$B\in\ZZ^{n\times n}$. If $B'$ is mutation equivalent~\cite[Def.~4.2]{FZ1} 
to $B$, the cluster algebra $\cA(B')$ is isomorphic to $\cA(B)$.

The cluster algebras we will study in this article are those which are given by
 a matrix $B$ whose associated diagram is one in Figure~\ref{fig:ellipticR} 
(the symbol close to the diagram denotes the associated elliptic root system, 
see Point (c) below). We call these four cluster algebras \emph{tubular}, 
due to their categorifaction by a tubular cluster category which we
discuss in the present paper.

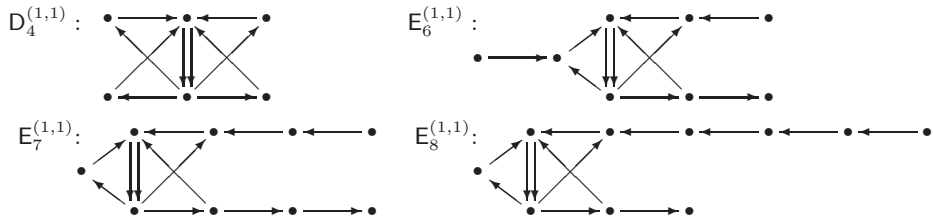
\begin{figure}[h]
\begin{center}
  \begin{picture}(300,40)
    \put(10,20){
      \put(-10,15){\RVCenter{\small $\sD_4^{(1,1)}:$}}
      \multiput(0,-15)(30,0){3}{\circle*{3}}
      \multiput(0,15)(30,0){3}{\circle*{3}}
      \multiput(26,-15)(30,30){2}{\vector(-1,0){22}}
      \multiput(4,15)(30,-30){2}{\vector(1,0){22}}
      \multiput(28.5,11)(3,0){2}{\vector(0,-1){22}}
      \multiput(3,-12)(30,0){2}{\vector(1,1){24}}
      \multiput(27,-12)(30,0){2}{\vector(-1,1){24}}
    }
    \put(150,20){
      \put(0,15){\RVCenter{\small $\sE_6^{(1,1)}:$}}
      \multiput(0,0)(30,0){2}{\circle*{3}}
      \multiput(50,-15)(0,30){2}{
      	\multiput(0,0)(30,0){3}{\circle*{3}}
      }
      \multiput(76,15)(30,0){2}{\vector(-1,0){22}}
      \multiput(54,-15)(30,0){2}{\vector(1,0){22}}
      \put(4,0){\vector(1,0){22}}
      \multiput(48.5,11)(3,0){2}{\vector(0,-1){22}}
      \put(46.4,-12.3){\vector(-4,3){12}}
      \put(53,-12){\vector(1,1){24}}
      \put(34.6,2.7){\vector(4,3){12}}
      \put(77,-12){\vector(-1,1){24}}
    }
  \end{picture}
\end{center} 
\begin{center}
  \begin{picture}(300,40)
    \put(0,20){
      \put(0,15){\RVCenter{\small $\sE_7^{(1,1)}$:}}
      \put(0,0){\circle*{3}}
      \multiput(20,-15)(30,0){4}{\circle*{3}}
      \multiput(20,15)(30,0){4}{\circle*{3}}
      \multiput(46,15)(30,0){3}{\vector(-1,0){22}}
      \multiput(24,-15)(30,0){3}{\vector(1,0){22}}
      \multiput(18.5,11)(3,0){2}{\vector(0,-1){22}}
      \put(16.4,-12.3){\vector(-4,3){12}}
      \put(23,-12){\vector(1,1){24}}
      \put(4.6,2.7){\vector(4,3){12}}
      \put(47,-12){\vector(-1,1){24}}
    }
    \put(150,20){
      \put(0,15){\RVCenter{\small $\sE_8^{(1,1)}$:}}
      \put(0,0){\circle*{3}}
      \multiput(20,-15)(30,0){3}{\circle*{3}}
      \multiput(20,15)(30,0){6}{\circle*{3}}
      \multiput(46,15)(30,0){5}{\vector(-1,0){22}}
      \multiput(24,-15)(30,0){2}{\vector(1,0){22}}
      \multiput(18.5,11)(3,0){2}{\vector(0,-1){22}}
      \put(16.4,-12.3){\vector(-4,3){12}}
      \put(23,-12){\vector(1,1){24}}
      \put(4.6,2.7){\vector(4,3){12}}
      \put(47,-12){\vector(-1,1){24}}
    }
  \end{picture}
\end{center} 
\caption{Quivers associated to some elliptic root systems}
\label{fig:ellipticR} 
\end{figure}

In~\cite{FZ2} Fomin and Zelevinsky showed that a cluster algebra
$\cA(B)$ is of finite type if and only if $B$ is mutation equivalent to 
the skew-symmetrization of a Cartan matrix of finite type. 
Moreover, they obtained in this case
a natural bijection between the cluster variables and almost positive roots in 
the corresponding root system. A broader class of cluster algebras are those
with a mutation finite exchange matrix. Each cluster algebra associated 
with the 
signed adjacency matrix of arcs of a triangulation on a (possibly bordered) 
two-dimensional surface, studied by Fomin, Shapiro and Thurston~\cite{FST}, 
belongs to this class. Note that in this case cluster variables are 
parametrized by tagged arcs on the corresponding surface.
In a recent paper by Felikson, Shapiro and 
Tumarkin~\cite{FeShTu08} it is shown that  there are
besides the above mentioned family of mutation finite exchange matrices 
only 11 further (exceptional) mutation classes of skew-symmetric exchange  
matrices of rank $\geq 3$, namely those given by: 
\begin{itemize}
\item[(a)]  
Cartan matrices of type $\sE_6$, $\sE_7$, $\sE_8$. 
The corresponding  cluster algebras are  of finite type.
\item[(b)] 
generalized Cartan matrices of type $\tilde{\sE}_6$, $\tilde{\sE}_7$, 
$\tilde{\sE}_8$. The corresponding
cluster algebras are categorified by the cluster category of a 
tame hereditary algebra of the corresponding type, and cluster variables
correspond naturally to almost positive Schur roots.
\item[(c)]
certain orientations of the diagrams describing elliptic root systems
of type $\sE_6^{(1,1)}, \sE_7^{(1,1)}$ resp.~$\sE_8^{(1,1)}$, see
Figure~\ref{fig:ellipticR} above.
\item[(d)] the Diagrams $\mathsf{X}_6$ and $\mathsf{X}_7$ recently discovered
by Derksen and Owen~\cite{DO08}.
\end{itemize}

Note also that the exchange matrix associated to the quiver of type 
$\sD_{4}^{(1,1)}$ in Figure~\ref{fig:ellipticR} is  mutation finite. In fact, 
it corresponds in the setup of~\cite{FST} to the 2-sphere with 4 punctures.

We  
give in this paper a uniform categorification of 
the cluster algebras associated to the Diagrams in 
Figure~\ref{fig:ellipticR}. To this end, 
we denote by $\coh \XX$ the category of coherent sheaves over a 
weighted projective line $\XX$ in the sense of Geigle-Lenzing \cite{GL87}, 
see also Section~\ref{sec:coh}. 
The orbit category $\cC_\XX :=\cDb(\coh\XX)/\euler{\tau^{-1}[1]}$, called the \emph{cluster category} associated to $\XX$,
of the derived category $\cDb(\coh\XX)$ is a 
triangulated 2-Calabi-Yau category \cite{Ke05}
admitting a cluster structure in the sense of \cite{BIRS}. Moreover, it
is easy to see that $\coh\XX$ and $\cC_\XX$ have the same Auslander-Reiten 
quiver. An object $X$ in $\cC_\XX$ is called \emph{rigid} if $\Ext^1_\cC(X,X)=0$.

The weighted projective line $\XX$ comes with a weight sequence 
$\bp=(p_1,\ldots,p_t)$. The complexity 
of $\coh\XX$ depends essentially on the value of the
virtual genus 
$\genus{\XX}=1+\frac{1}{2}\left((t-2)p-\sum_{i=1}^t \frac{p}{p_i}\right)$,
where $p=\lcm(p_1,\ldots,p_t)$. More precisely, $\coh\XX$ is tame if
and only if $\genus{\XX}\leq 1$. 
If $\genus{\XX}=1$ the weighted projective line $\XX$ is said to be of 
\emph{tubular type} and we call the category $\cC_\XX$ 
\emph{tubular}. 
In this case each connected component of the Auslander-Reiten 
quiver of $\cC_\XX$ is a tube. The following is the main result of this paper.

\begin{theorem}
\label{thm:main1}
  Each  cluster algebra $\cA$ of type $\sD_4^{(1,1)}, \sE_6^{(1,1)}, 
  \sE_7^{(1,1)}$ resp.~$\sE_8^{(1,1)}$ can be categorified by $\cC_\XX$, where 
  $\XX$ is of tubular type with weight sequence $\bp=(2,2,2,2)$, $(3,3,3)$, 
  $(4,4,2)$ 
  and $(6,3,2)$ respectively. More precisely, there exists a cluster character 
  in the sense of Palu \cite{Pa08}, which induces a bijective map from the set 
  of  (isomorphism classes of) indecomposable 
  objects $E$ of $\cC_\XX$ which are rigid to the set of cluster variables of 
  $\cA$ 
  and a bijection between the (isomorphism classes of) basic 
  cluster-tilting objects and  the clusters.
\end{theorem}

\begin{remark}\ \
\label{rem:main}
{\bf (a)}\ 
The parametrization of the indecomposable rigid objects in $\cC_\XX$ 
reduces to the description of positive real Schur roots in the 
Grothendieck group $\KX$ of $\coh\XX$. In the tubular case this can be
done in a closed form.
The condition for two 
indecomposable rigid objects of being Ext-orthogonal can also be translated 
into a  handy criterion in terms of Schur roots.
This facilitates a combinatorial description of the 
exchange graph and the cluster complex. 
We shall carry out these description in a forthcoming paper \cite{BG2}.

{\bf (b)}
For the proof in the D-case we produce an explicit bijection 
between positive real Schur roots and tagged arcs on the 2-sphere with 
4 punctures, which respects  the respective notions of compatibility. 
See Proposition~\ref{thm:bijection} for the precise statement.
We believe that this bijection is of independent interest.
\end{remark}

The following result is an immediate consequence of Theorem~\ref{thm:main1}.

\begin{corollary}
The cluster complex of each tubular cluster algebra $\cA$ is the 
clique complex of the compatibility relation.
\end{corollary}

In the $\sE$-cases we obtain some addtional information from the proof of
our result.
\begin{corollary}
\label{cor:1.4}
Each tubular cluster algebra $\cA$ of type $\sE_6^{(1,1)},   \sE_7^{(1,1)}$ or 
$\sE_8^{(1,1)}$ is finitely generated and the cluster 
monomials form a linearly   independent family.
\end{corollary}

In fact we show in section \ref{ssec:TubE} that these cluster algebras are instances of the cluster algebras categorified in~\cite{GLS08}, see also~\cite{GLS10}. 
All cluster algebras of this kind have the
above mentioned properties, see Section~\ref{ssec:GLS}.

%%%%%%%%%%%%%%%%%%%%%%%%%%%%%%%%%%%%%%%%%%%%%%%%%%
%%%%%%%%%%%%%%%%%%%%%%%%%%%%%%%%%%%%%%%%%%%%%%%%%%

\section{Coherent sheaves over weighted projective lines}
\label{sec:coh}
\subsection{Basic Notions}
Let $k$ be an algebraically closed field and denote by $\PP^1$ the projective
line over $k$. Let $\blam=(\lam_1,\ldots,\lam_t)$ be a collection of
pairwise different points of $\PP^1$ and $\bp=(p_1,\ldots,p_t)$ a weight
sequence with $p_i\in\NN_{\geq 2}$. Then the tuple
$\XX=(\PP^1,\blam,\bp)$ is called a weighted projective line.
The weight function is defined by

\[
p_\lam\df\PP^1\rightarrow\NN,\mu\mapsto\begin{cases}
p_i &\text{if } \mu=\lam_i \text{ for some } i\\
1   &\text{else.}\end{cases} 
\]
For later use we set $p:=\lcm(p_1,\ldots,p_t)$.

Geigle and Lenzing~\cite{GL87} associate to $\XX$ a category of coherent 
sheaves as follows: Let $L(\bp)$ be the rank 1 additive group
\[
L(\bp):= \euler{\vx_1,\ldots,\vx_t\mid p_1\vx_1=\cdots=p_t\vx_t=\vc}
\]
and $S(\bp,\blam)$ the $L(\bp)$-graded commutative algebra
\[
S(\bp,\blam) = k[u,v,x_1,\ldots,x_t]/
(x_i^{p_i}-\lam'_iu-\lam''_iv\mid i=1,\ldots,t) 
\]
where $\deg x_i=\vx_i$ and $\lam_i=[\lam'_i:\lam''_i]\in\PP^1$.
Now, $\coh\XX$ is the quotient category of finitely generated 
$L(\bp)$-graded $S(\bp,\blam)$-modules by the Serre subcategory of
finite length modules.

Geigle and Lenzing showed that $\coh\XX$ is a hereditary abelian category with
finite dimensional $\Hom$ and $\Ext$ spaces. The free module gives a structure
sheaf $\cO$, and shifting the grading gives twists $E(\vx)$ for any sheaf
$E\in\coh\XX$ and $\vx\in L(\bp)$. Moreover, they showed that, putting
$\vn := \sum_{i=1}^t(\vc-\vx_i)-2\vc$  the following version of
Serre duality
\[
D\Ext^1_\XX(E,F)\cong\Hom_\XX(F,E(\vn))
\]
holds. As a consequence, $\coh\XX$ has Auslander-Reiten sequences with the
shift $E\mapsto E(\vn)$ acting as Auslander-Reiten translation.

%%%%%%%%%%%%%%%%%%%%%%%%%%%%%%%%%%%%%%%%%%%%

Every sheaf is a direct sum of a `vector bundle' which has a filtration
with factors of the form $\cO(\vx)$, and a sheaf of finite length. The 
indecomposable sheaves of finite length are readily described:
For each $\mu\in\PP^1\setminus\blam$ there is a unique simple sheaf $S_\mu$
`concentrated at $\mu$' and for $\lam_i\in\PP^1$ there are simple sheaves
$S_{i,1},\ldots, S_{i,p_i}$ `concentrated at $\lam_i$' and the only non-trivial 
extension between them are
\[
\Ext^1_\XX(S_\mu,S_\mu)\cong k \text{ and } \Ext^1_\XX(S_{i,j},S_{i,j'})\cong k
\text{ if } j-j'\equiv 1\pmod{p_i}.
\]
As a consequence, for each simple sheaf $S$ and $l\in\NN$ there exists a unique
indecomposable sheaf $S^{(l)}$ of length $l$ and socle $S$, and up to 
isomorphism
all indecomposable sheaves of finite length are of this form. Summarizing:

\begin{proposition} \label{coh:prp0} 
The category $\coh_0\XX$ of sheaves of finite length is an exact abelian, 
uniserial subcategory of $\coh\XX$ which is stable under Auslander-Reiten 
translation. The components of the Auslander-Reiten quiver of $\coh_0\XX$
form a family of standard tubes $(\cT_\mu)_{\mu\in\PP^1}$ with rank
$\rk\cT_\mu=p_\lam(\mu)$, see \cite{Ri84} for definitions.
\end{proposition}

%%%%%%%%%%%%%%%%%%%%%%%%
\subsection{Discrete invariants}
The Grothendieck group $\KX$ of $\coh\XX$ is a free abelian group of 
rank $n={2+\sum_{i=1}^t(p_i-1)}$. 
Since $\coh\XX$ is hereditary, the 
homological form 
\[
\euler{E,F} = \dim\Hom_\XX(E,F)-\dim\Ext^1_\XX(E,F)
\]
descends to an integral bilinear form on $\KX$. For the same reason,
the Auslander-Reiten translate induces a linear transformation $\tau$ on
$\KX$ such that $[E(\vn)]=\tau[E]$ for all $E\in\coh\XX$. 
From Serre duality follows
\[
\euler{\tau\be,\tau\bff}=\euler{\be,\bff}=-\euler{\bff,\tau\be}
\]
for all $\be,\bff\in \KX$. Thus $\tau\be=\be$ implies $\euler{\be,\be}=0$.

There exists $\bh_\infty\in\KX$ such that $\bh_\infty=[S_\mu]$ for all 
$\mu\in\PP^1\setminus\blam$. We define
the \emph{rank} of a sheaf $E$ by
$$
\rk E=\euler{[E],\bh_\infty}.
$$
It is easy to see that
$\rk \cO(\vx)=1$ for all $\vx\in L(\bp)$ and $\rk S=0$ for all simple sheaves.

Next, putting $\bh_0 :=\sum_{k=0}^{p-1}[\cO(k\vn)]$ we may define the 
\emph{degree} of $E$ as
\[
\deg E =\euler{\bh_0,[E]}-(\rk E)\euler{\bh_0,[\cO]},
\]
and it turns out that
$\deg S=p/p_\lam(\mu)$ if $S$ is a simple concentrated at $\mu\in\PP^1$,
and $\deg\cO(\vx)=\del(\vx)$ with
$\del\df L(\bp)\rightarrow \ZZ$ defined by $\del(\vx_i)=p/p_i$.

Clearly, rank and degree are linear functionals on $\KX$. 
Thus, we may define for $\be\in\KX$ 
\[
0\prec\be :\Iff \rk\be>0 \text{ or } (\rk\be=0 \text{ and } \deg\be>0). 
\]
This converts $\KX$ into an  ordered group and we say  $\be$ is
{\em positive} if $0\prec\be$.
Note that for each $E\in\coh\XX$ we
have $0\prec [E]$.

\begin{definition} Let $\XX$ be a  weighted projective line. We say
that $\be\in \KX$ is a {\em positive root} if there exists an 
{\em indecomposable} $E\in\coh\XX$ with $[E]=\be$. Such a root 
is called a {\em Schur root} if $E$ can be chosen such that 
$\End_\XX(E)\cong k$, it is called {\em real} if $\euler{\be,\be}=1$ and
{\em isotropic} if $\euler{\be,\be}=0$.
\end{definition}

Recall that a sheaf $E\in\coh\XX$ is called \emph{rigid} if
$\Ext^1_\XX(E,E)=0$. Remarkably, we have the following alternative
characterization of Schur roots~\cite[Prop.~4.4.1]{Mel04}.

\begin{proposition}
\label{prop:real-pos-Schur}
The map $E\mapsto [E]$ induces a bijection between
the isomorphism classes of indecomposable rigid sheaves
and the real positive Schur roots.
\end{proposition}

%%%%%%%%%%%%%%%%%%%%%%%%
\subsection{Stability} Let $\QQi=\mathbb{Q}\cup\{\infty\}$.
For each positive $\be\in\KX$  we define 
its \emph{slope} by $\slo(\be)=\frac{\deg\be}{\rk\be}\in\QQi$. If $E\in\coh\XX$ we write
$\slo(E)=\slo([E])$.

We say that $E$ es \emph{stable} (resp.~\emph{semistable}) if for each non-trivial subbundle
$E'$ of $E$ we have $\slo(E')<\slo(E)$ (resp.~$\slo(E')\leq\slo(E)$).
We have the following result~\cite[Prop.~5.2]{GL87}:

\begin{proposition} \label{coh:prp-stab}
For each $q\in\QQi$ let $\cC_q$ be the subcategory of $\coh\XX$ consisting
of all semistable coherent sheaves of slope $q$. In particular, $\cC_\infty$
is the subcategory of finite length sheaves.  Then the following holds:
\begin{itemize}
\item[(a)]
$\cC_q$ is an extension closed exact abelian finite length subcategory of
$\coh\XX$ with the simple objects being precisely the stable vector bundles.
\item[(b)]
$\Hom_\XX(\cC_q,\cC_{q'})=0$ if $q'<q$.
 \end{itemize}
\end{proposition}

%%%%%%%%%%%%%%%%%%%%%%%%%%%%%
\subsection{Tubular weighted projective lines}
The complexity of the classification of indecomposables in $\coh\XX$ depends
essentially on the value of the virtual genus
$\genus{\XX}=1+\frac{1}{2}\left((t-2)p-\sum_{i=1}^t \frac{p}{p_i}\right)$,
If $\genus\XX\leq 1$
the category
$\coh\XX$ is derived equivalent to the module category of a tame hereditary 
algebra. For $\genus{\XX}>1$ the classification problem is known to be wild.
It is elementary to see that $\genus{\XX}=1$ if and only if
\[
\bp\in\{(6,3,2),\ (4,4,2),\ (3,3,3),\ (2,2,2,2)\}.
\]
In this case $\XX$ is called \emph{tubular}. We have then
$p\vn=0$ so that $\euler{\bh_0,[\cO]}=0$ and the degree simplifies to
$\deg\be=\euler{\bh_0,\be}$. The type $(2,2,2,2)$
(resp. $(3,3,3)$, $(4,4,2)$ and $(6,3,2)$) is closely related to the
Dynkin diagram $\sD_4$ (resp. $\sE_6$, $\sE_7$ and $\sE_8$), see
Remark \ref{rem:related-type}, and therefore we shall say that $\XX$ is of
\emph{$\sD$-type} (resp. of \emph{$\sE$-type}).

Lenzing and his collaborators~\cite[Thm.~5.6]{GL87}
and~\cite[Thm.~4.6]{LM93},  showed the following fundamental result:
\begin{theorem} \label{coh:thm-tub}
Let $\XX$ be a tubular weighted projective line. Then:
\begin{itemize}
\item[(a)] 
For any $q\in\QQi$ the subcategory $\cC_q\subset\coh\XX$ 
(see Proposition~\ref{coh:prp-stab}) is stable under Auslander-Reiten 
translation and it is equivalent to $\cC_\infty$.
\item[(b)]
If $E\in\coh\XX$ is indecomposable then $E\in\cC_q$ for some $q\in\QQi$ and
$[E]\in \KX$ is a positive real or isotropic positive root.
\item[(c)] If $0\prec \be\in\KX$ and $\euler{\be,\be}\in\{0,1\}$ then 
$\be$ is a real or isotropic root.
\item[(d)] 
If $\be\in \KX$ is a positive real root there exists up to isomorphism
a unique indecomposable $E\in\coh\XX$ with $[E]=\be$. If $\bff\in \KX$
is a positive isotropic root, there exists a $\PP^1$-family of indecomposable
sheaves $(F_\mu)_{\mu\in\PP^1}$ with $[F_\mu]=\bff$.
\end{itemize}
\end{theorem} 

\begin{remark}
If $\genus{\XX}<1$ it is well-known that similar statements 
to (b), (c) and~(d) hold true. For $\genus{\XX}>1$ a description of the
positive roots, paralleling Kac's theorem for representations of quivers,
was found recently by Crawley-Boevey~\cite{CB05}. 
\end{remark}

%%%%%%%%%%%%%%%%%%%%%%%%
\subsection{Positive Schur roots in the tubular case}
Let $\XX$ be a tubular weighted projective line.  
\begin{definition}
For $q\in\QQi$ we define 
$a(q)\in\ZZ$ and  $b(q)\in\ZZ_{\geq 0}$ such that $\gcd(a(q),b(q))=1$ and 
$q=\frac{a(q)}{b(q)}$. Furthermore we define
$\bh_q=b(q)\bh_0 +a(q)\bh_\infty$.
\end{definition}
 
Note that $\bh_q$ is a positive isotropic root with $\slo(\bh_q)=q$.

For $\be\in\KX$ we write $\tau^\ZZ\be=\{\tau^i\be\mid i\in\ZZ\}$. Since
$\XX$ is tubular $\tau^\ZZ\be$ is a finite set of cardinality
$\abs{\tau^\ZZ\be}$ which is a divisor of $p$ (recall $p\vn=0$), and
$
\Sig\tau^\ZZ\be := \sum_{\be'\in\tau^\ZZ\be}\be'
$
is well-defined.

\begin{lemma} For $0\prec\be\in\KX$ with $\slope(\be)=q\in\QQi$ there exists
a positive integer $\ql(\be)$ such that $\Sig\tau^\ZZ\be =\ql(\be)\bh_q$.
\end{lemma}

\begin{proof}
Since $\XX$ is tubular, $\slo(\be)=\slo(\tau\be)=q$, 
thus $\slo(\Sig\tau^\ZZ\be)=q$. Moreover $\Sig\tau^\ZZ\be$ is positive and
$\tau(\Sig\tau^\ZZ\be)=\Sig\tau^\ZZ\be$, so 
$\euler{\Sig\tau^\ZZ\be,\Sig\tau^\ZZ\be}=0$. Our claim follows now from
\cite[Lemma 2.6]{LM93}.
\end{proof}

\begin{proposition}\label{prp:SchurRt}
Let $\XX$ be a tubular weighted projective line. The positive isotropic
Schur roots are precisly of the form $\bh_q$ with $q\in\QQi$. Moreover,
for $0\prec\be\in\KX$ the following are equivalent:
\begin{itemize}
\item[(a)] $\be$ is a real positive Schur root,
\item[(b)] there exists an indecomposable $E\in\coh\XX$ with 
$\Ext^1_\XX(E,E)=0$ and $[E]=\be$,
\item[(c)] $\euler{\be,\be}=1$ and $\ql(\be)<\abs{\tau^\ZZ\be}$,
\end{itemize}
\end{proposition}

\begin{proof}
The equivalence of (a) and (b) was stated in 
Proposition~\ref{prop:real-pos-Schur}.
By Proposition~\ref{coh:prp0} and Theorem~\ref{coh:thm-tub}(a) the 
Auslander-Reiten quiver of $\coh\XX$ consists only of standard tubes. It is
well known, that any indecomposable $E$ in a standard tube of rank $r$ 
has trivial endomorphism rings if and only if the quasi-length of $E$ is less
or equal to $r$. 

Now, $\bh_q$ is by \cite[Lemma 2.6]{LM93} and Theorem~\ref{coh:thm-tub}(c) 
the smallest positive isotropic root of slope $q$. 
By the above observation it is a Schur root and all other isotropic roots of 
slope $q$ are not Schur. 

It remains to show that in case $[E]$ is a real root
$\ql([E])$ is the quasi-length of $E$. This follows 
from~\cite[Thm.~4.6(iv)]{LM93} applied to the family representing $\bh_q$.
\end{proof}

%%%%%%%%%%%%%%%%%%%%%%%%
\subsection{Ext-orthogonality in the tubular case}
Let $\XX$ be a tubular weighted projective line and $E$, $F$ be two 
indecomposable rigid sheaves. We call $E$ and $F$ \emph{Ext-orthogonal} or 
\emph{compatible} if $E\oplus F$ is rigid. We denote by $\be=[E]$ and 
$\bff=[F]$ their classes in the Grothendieck group. 
We will show that the condition for $E$ and $F$ to be Ext-orthogonal 
can be expressed  only in terms of $\be$ and $\bff$  and the homological 
form.

First of all, $\slope(\be)$ and $\slope(\bff)$ are calculated using
the homological form.  If $\slope(\be)<\slope(\bff)$ then $E\oplus F$
is rigid if and only if $\euler{\bff,\be}=0$.  It remains to consider
the case where $q:=\slope(\be)=\slope(\bff)$. Since the indecomposable
sheaves with slope $q$ form a family of pairwise orthogonal tubes
$E\oplus F$ is rigid whenever $E$ and $F$ belong to different
tubes. The latter happens if and only if $\euler{\be,\tau^{i}\bff}=0$
for all $i=0,\ldots,r-1$ where $r=\abs{\tau^\ZZ\bff}$.

It remains to characterize when two indecomposable rigid sheaves of
the same tube are Ext-orthogonal. For this we recall that each tube is
an abelian category whose quasi-simples form an orbit $\tau^\ZZ S$
under the Auslander-Reiten translation $\tau$ of a single quasi-simple
object $S$.  Now, each indecomopsable rigid sheaf $E$ defines
$\base{E}$, the set of quasi-simples which occur as composition factor
of $E$, more precisely, if $i$ is minimal with $\tau^i S\in\base{E}$
then $\base{E}=\{\tau^{i+j} S\mid j=0,\ldots,\ql(\be)-1\}$. It is not
hard to see that $E$ and $F$ are Ext-orthogonal if and only if one of
the following three conditions is satisfied: (i) $\base{E}\subseteq
\base{F}$, (ii) $\base{F}\subseteq\base{E}$ or (iii) $\base{E}\cap
\tau^h \base{F}=0$ for $h=-1,0,1$.

In the following we give equivalent conditions on the vectors $\be$
and $\bff$ for each of these cases. We first show how the values of
$\euler{\bff,-}$ and $\euler{-,\bff}$ vary in the tube. We
have shown this in Figure \ref{fig:ext-orthogonality} where the bottom
dotted line indicates the $\tau$-orbit of the quasi-simples and the
dotted line on the top the $\tau$-orbit of indecomposable rigid
objects of maximal quasi-length. We have indicated the values as
vectors
$\left[\begin{matrix}\euler{-,\bff}\\ 
\euler{\bff,-}\end{matrix}\right]$.
The region of indecomposable rigid objects which are Ext-orthogonal to
$F$ is shown gray.
\begin{figure}[!h]
\begin{center}
  \includegraphics[scale=1,viewport=140 566 470 720,clip]{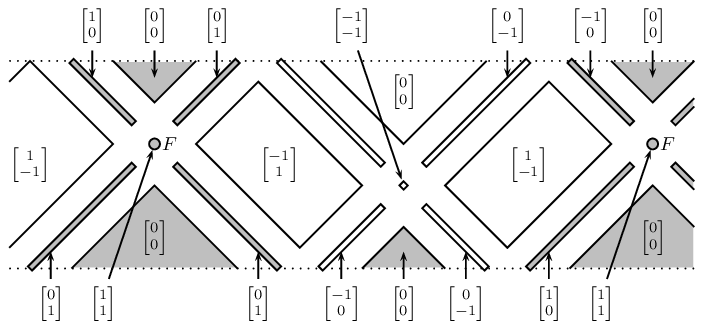}
\end{center}
\caption{Ext-orthogonality in a tube}
\label{fig:ext-orthogonality} 
\end{figure}

 Note that in the cases (i) and (ii) we have $\euler{\be,\bff}\geq 0$
 and $\euler{\bff,\be}\geq 0$ and in case (iii) we have
 $\euler{\be,\bff}=0=\euler{\bff,\be}$.  Now, if $\euler{\be,\bff}>0$
 and $\euler{\bff,\be}=0$ (or $\euler{\be,\bff}=0$ and
 $\euler{\bff,\be}>0$) then $E$ and $F$ lie on the same ray or coray
 and (i) or (ii) is satisfied. If
 $\euler{\be,\bff}=0=\euler{\bff,\be}$ then let $j$ be minimal with
 $\euler{\tau^j\be,\bff}\neq 0$. If $\euler{\tau^j \be,\bff}>0$ then
 again (i) or (ii) is satisfied whereas if $\euler{\tau^j \be,\bff}<0$
 then (iii) holds if and only if $\ql(\be)+\ql(\bff)<\abs{\tau^\ZZ
   \bff}$.

Altogether we have proved the following statement.
\begin{proposition}
  Let $E$ and $F$ be two indecomposable rigid sheaves with classes
  $\be=[E]$ and $\bff=[F]$. Then $E$ and
  $F$ are Ext-orthogonal if and only if one of the following
  conditions is satisfied.
  \begin{itemize}
  \item[{\rm (a)}]
    $\slope(\be)<\slope(\bff)$ and $\euler{\bff,\be}=0$,
  \item[{\rm (b)}]
    $\slope(\be)>\slope(\bff)$ and $\euler{\be,\bff}=0$,
  \item[{\rm (c)}] $\slope(\be)=\slope(\bff)$ and
    $\euler{\tau^j\be,\bff}=0$ for $j=0,\ldots,\ql(\be)-1$,
  \item[{\rm (d)}] $\slope(\be)=\slope(\bff)$ and $\euler{\be,\bff}\geq 0$, $\euler{\bff,\be}\geq 0$ but not both zero,
  \item[{\rm (e)}] $\slope(\be)=\slope(\bff)$ and
    $\euler{\be,\bff}=0=\euler{\bff,\be}$ and there exists a $j$ such
    that $\euler{\tau^j\be,\bff}\neq 0$. If this $j$ is minimal then 
    either
    \begin{itemize}
    \item[{\rm (e1)}] $\euler{\tau^j\be,\bff}>0$, or
    \item[{\rm (e2)}] $\euler{\tau^j\be,\bff}<0$ and $\ql(\be)+\ql(\bff)<\abs{\tau^\ZZ\be}$.
    \end{itemize}
  \end{itemize}
\end{proposition}

%%%%%%%%%%%%%%%%%%%%%%%%
\subsection{Rigid objects and cluster-tilting objects}
\label{subsec:rigid-cc}

A rigid sheaf $E\in\coh \XX$ is called a \emph{tilting sheaf} (or \emph{maximal
  rigid}) if it is maximal among the rigid sheaves in the following
sense:  
any sheaf $F$ such that $E\oplus F$ is again rigid belongs
to the additive closure $\Add E$.  Note that the number of
(pairwise non-isomorphic) indecomposable direct summands of any tilting sheaf
equals the rank $n={2+\sum_{i=1}^{t}(p_i-1)}$ of the Grothendieck
group.  A rigid sheaf with $n-1$ (pairwise non-isomorphic) indecomposable
direct summands is called an \emph{almost complete tilting sheaf}. 
An indecomposable sheaf $E'$ such that $E\oplus E'$ is a tilting sheaf is
called a {\em complement} of the almost complete tilting sheaf $E$.
Two indecomposable rigid objects $E$ and $E'$ are called \emph{compatible}
if $E\oplus E'$ is rigid.
 
Following Keller~\cite{Ke05}, the orbit category
$\cC_\XX:=\cDb(\coh\XX)/\euler{\tau^{-1}[1]}$ associated to a weighted projective
line is a triangulated 2-Calabi-Yau category. 

\begin{remark}
\label{rem:tilt-clutilt}
As explained in~\cite{BKL08},
the composition of the canonical functors
\[
\coh\XX\xrightarrow{\rm incl.}\cDb(\coh\XX)\xrightarrow{\rm proj.}\cC_\XX
\]
allows to think of $\coh\XX$ as a non-full subcategory of $\cC_\XX$ which has
the same isoclasses of indecomposable resp.~rigid objects. 
It follows that the tilting objects in $\coh \XX$ correspond bijectively 
with the \emph{cluster-tilting objects} in $\cC_\XX$, that is, the maximal rigid objects in $\cC_\XX$.

By~\cite[Prop.~5.14]{Hueb96} each almost complete tilting sheaf in 
$\coh\XX$ has precisely two complements. It follows that the exchange graph for
tilting objects in $\coh\XX$ and the exchange graph for cluster tilting
objects can be identified.
By~\cite{BKL08}, the cluster category
$\cC_\XX$ has a cluster structure in the sense of~\cite{BIRS}.

By Proposition~\ref{prop:real-pos-Schur}  
the positive real Schur roots 
parametrize the indecomposable rigid objects in~$\cC_\XX$.
\end{remark}

%%%%%%%%%%%%%%%%%%%%%%%%%%%%%%%%%%%%%%%%%%%%%%%%%%
%%%%%%%%%%%%%%%%%%%%%%%%%%%%%%%%%%%%%%%%%%%%%%%%%%
\section{Tubular cluster categories}

In case the weighted projective line $\XX$ is tubular  
we call $\cC_\XX$ the corresponding
{\em tubular cluster category}.  Moreover, we note, that 
by~\cite[Thm.~8.8]{BKL08}
in this case the exchange graph of cluster-tilting objects is connected.

\subsection{GLS-cluster categories and -character} \label{ssec:GLS}
Let us review some of the results of~\cite{Am08} and~\cite{GLS08}. 
Note that in \cite{GLS10} a more general theory was developed. 
However, the description in~\cite{GLS08} is more convenient for our purpose:  
Let $Q$ be a quiver without oriented cycles and $n$ vertices and denote by 
$\CC Q$ the path algebra of $Q$.
Moreover, let $M_1, \ldots, M_r$ be a family of indecomposable, pairwise 
non-isomorphic preinjective representations of $Q$, closed under successors 
and such that $M=\oplus_{i=1}^r M_i=\oM\oplus I$ for an injective cogenerator 
$I$ of $\lmd{\CC Q}$.
Note, that $E_\oM=\End_Q(\oM)$ is a basic algebra of
global dimension~2. The Gabriel quiver $\tQ_\oM$ of $E_\oM$ is given
by the full subquiver of the preinjective component of $\CC Q$ which is
supported in the summands of $\oM$.

Let $\Lam$ be the  preprojective algebra associated to the path algebra $\CC Q$.
Since  $\CC Q$ is a subalgebra of $\Lam$
we have the restriction functor 
${?\rvert_Q\df\lmd{\Lam}\rightarrow\lmd{\CC Q}}$. 
Recall from \cite{HRS96} that a finite dimensional algebra $A$ is called 
\emph{piecewise hereditary} if $\cDb(\lmd{A})$ is triangle-equivalent to 
$\cDb(\Her)$ for $\Her$ a connected abelian hereditary $k$-category 
and a tilting object. Note that by \cite{Ha2001}, the 
only relevant cases are where $\Her=\lmd{H}$ for some hereditary algebra $H$ 
or $\Her=\coh\XX$ for some weighted projective line $\XX$.

\begin{theorem} For the subcategory
$\cC_M:=\{X\in\lmd{\Lam}\mid X\rvert_Q\in\Add(M)\}$ the following holds:
\begin{itemize}
\item[(a)]
$\cC_M$ is  Frobenius category. The stable category $\scC_M$
is a triangulated 2-Calabi-Yau category with a  basic cluster-tilting
object $T_M=\oplus_{i=1}^{r-n}T_i$ such that the quiver $\hQ_\oM$ of 
$\End_{\scC_M}(T_M)$ is obtained from
$\tQ_\oM$ by inserting extra arrows $M_i\rightarrow M_j$ whenever 
$M_j\cong\tau_Q M_i$.
\item[(b)]
We have a cluster character 
$\vph_?\df\operatorname{obj}(\scC_M)\rightarrow A_M$ 
in the sense of~\cite[Def.~2]{Pa08}
with the following additional properties:
\begin{itemize}
\item[(i)] $A_M$ is a finitely generated cluster algebra with trivial
coefficients and initial seed $((\vph_{T_i})_{i=1,\ldots,r-n},\hQ_\oM)$.
\item[(ii)]
The family $(\vph_X)$ where $X$ runs over the isoclasses of rigid objects in
$\scC_M$, is linearly independent in $A_M$.
\end{itemize}
\item[(c)]
If $E_\oM$ is piecewise hereditary, then 
$\scC_M$ is triangle equivalent to the cluster category
$\cDb(\lmd{E_\oM})/\euler{\tau_\cD^{-1}[1]}$.
\end{itemize}
\end{theorem}

\begin{remark} \label{rem:A-GLS}
Part~(a) of the theorem is a consequence of Theorems~2.1, 2.2
and~2.3 in~\cite{GLS08}. Part~(b) follows easily from the discussion in the 
Sections 3.2, 3.4. and~3.6 of~\cite{GLS08}. Part~(c) follows 
from~\cite[Theorem~5.15]{Am08}. Note, that Amiot's category $\cC_A$ is
by construction the triangulated hull of the orbit category 
$\cDb(\lmd{E_\oM})/\euler{\tau_\cD^{-1}[1]}$, however if $A$ is piecewise
hereditary,  the orbit category is already triangulated by~\cite{Ke05}.
\end{remark}

\begin{remark}
The above theorem implies the following: Suppose $E_\oM$ is derived equivalent
to $\coh\XX$ for a tubular weighted projective line $\XX$, 
then by~(c) the stable category $\scC_M$ is triangle equivalent to the cluster 
category $\cC_\XX$. It follows from~\cite[Thm.~8.8]{BKL08} that
in this case the exchange graph of the cluster 
tilting objects of $\scC_M$ is connected. Thus,
we obtain by~(b) a bijection between the positive 
real Schur roots and cluster variables. We will see in the next subsection that
this situation is given for the tubular E-types $(3,3,3)$, $(4,4,2)$ 
and~$(6,3,2)$.

Moreover, two cluster variables 
belong to a common cluster if and only if the corresponding positive real 
Schur roots are Ext-orthogonal.
\end{remark}

%%%%%%%%%%%%%%%%%%%%%%%
\subsection{The tubular E-cases} \label{ssec:TubE}

We are now ready to give an explicit proof in each of the three E-cases that 
the tubular cluster algebra is categorified by $\coh \XX$ for $\XX$ a weighted 
projective line of weight type $(3,3,3)$, $(4,4,2)$ and $(6,3,2)$ respectively. 

\subsubsection{} Let $Q$ be the quiver as shown in the following figure on the 
left. The Auslander-Reiten quiver $\Gamma_Q$ of $Q$ has then the shape as 
shown in the middle of following picture.
\begin{center}
  \begin{picture}(300,50)
    \put(0,25){
      \put(-5,20){\RVCenter{\small $Q:$}}
      \put(0,0){\HVCenter{\small $1$}}
      \put(20,20){\HVCenter{\small $2$}}
      \put(20,0){\HVCenter{\small $3$}}
      \put(20,-20){\HVCenter{\small $4$}}
      \put(16,16){\vector(-1,-1){12}}
      \put(15,0){\vector(-1,0){10}}
      \put(16,-16){\vector(-1,1){12}}
    }
    \put(80,25){
      \put(0,0){
      \put(-5,20){\RVCenter{\small $\Gamma_Q:$}}
        \put(0,0){\HVCenter{\small $1$}}
        \put(20,20){\HVCenter{\small $2$}}
        \put(20,0){\HVCenter{\small $3$}}
        \put(20,-20){\HVCenter{\small $4$}}
        \put(4,4){\vector(1,1){12}}
        \put(5,0){\vector(1,0){10}}
        \put(4,-4){\vector(1,-1){12}}
        \put(24,16){\vector(1,-1){12}}
        \put(25,0){\vector(1,0){10}}
        \put(24,-16){\vector(1,1){12}}
      }
      \put(40,0){
        \put(0,0){\HVCenter{\small $5$}}
        \put(20,20){\HVCenter{\small $6$}}
        \put(20,0){\HVCenter{\small $7$}}
        \put(20,-20){\HVCenter{\small $8$}}
        \put(4,4){\vector(1,1){12}}
        \put(5,0){\vector(1,0){10}}
        \put(4,-4){\vector(1,-1){12}}
        \put(24,16){\vector(1,-1){12}}
        \put(25,0){\vector(1,0){10}}
        \put(24,-16){\vector(1,1){12}}
      }
      \put(80,0){
        \put(0,0){\HVCenter{\small $9$}}
        \put(20,20){\HVCenter{\small $10$}}
        \put(20,0){\HVCenter{\small $11$}}
        \put(20,-20){\HVCenter{\small $12$}}
        \put(4,4){\vector(1,1){12}}
        \put(5,0){\vector(1,0){10}}
        \put(4,-4){\vector(1,-1){12}}
      }
    }
    \put(230,0){
    \put(0,30){
      \put(0,0){
        \put(-5,20){\RVCenter{\small $\tQ_\oM:$}}
        \put(0,0){\HVCenter{\small $1$}}
        \put(25,25){\HVCenter{\small $2$}}
        \put(25,-6.25){\HVCenter{\small $3$}}
        \put(25,-25){\HVCenter{\small $4$}}
        \put(4,4){\vector(1,1){17}}
        \put(5,-1.25){\vector(4,-1){15}}
        \put(4,-4){\vector(1,-1){17}}
        \put(29,21){\vector(1,-1){17}}
        \put(30,-6){\vector(4,1){15}}
        \put(29,-21){\vector(1,1){17}}
      }
      \put(50,0){
        \put(0,0){\HVCenter{\small $5$}}
        \put(25,25){\HVCenter{\small $6$}}
        \put(25,-6.25){\HVCenter{\small $7$}}
        \put(25,-25){\HVCenter{\small $8$}}
        \put(4,4){\vector(1,1){17}}
        \put(5,-1.25){\vector(4,-1){15}}
        \put(4,-4){\vector(1,-1){17}}
      }
      \multiput(25,-25)(0,50){2}{
        \qbezier[20](6,0)(25,0)(46,0)
      }
      \multiput(0,0)(25,-6.25){2}{
        \qbezier[20](6,0)(25,0)(46,0)
      }
    }
    }
  \end{picture}
\end{center}
We choose $M=\oplus_{i=1}^{12}M_i$ the direct sum of all indecomposable modules.
Thus, $E_\oM$ is given by the quiver $\tQ_\oM$ with relations as shown in the 
previous picture on the right.
Use \cite[Thm. A]{BdP99} to check that this algebra is derived equivalent to a 
tubular algebra of type $(3,3,3)$. Note that condition (ii) can be verified 
computationally, see \cite{BdP99b}. 
The algebra $E_\oM$ is in fact  tubular as can be seen using the techniques of 
tubular extensions (which later also were called branch-enlargements) of tame 
concealed algebras described in \cite[Sect.~4.7]{Ri84}, although we don't need that. 
It follows that the bounded derived category of $E_\oM$ is triangle-equivalent 
to $\coh \XX$ for $\XX$ of weight type $(3,3,3)$ by \cite[Thm.~3.2, Prop.~4.1]{GL87}.
In particular $E_\oM$ is piecewise hereditary.
Recall that $\hQ_\oM$ denotes the quiver of $\End_{\cC_\XX}(\oM)$. A 
straightforward check shows that  
$\mu_2\mu_3\mu_4(\hQ_\oM)=\Delta$, where $\Delta$ is the quiver in 
Figure~\ref{fig:ellipticR} associated to $\sE_6^{(1,1)}$. 

\subsubsection{}Let $Q$ be the quiver as shown in the following figure on the 
left. The Auslander-Reiten quiver $\Gamma_Q$ of $Q$ has then the shape as 
shown in the following picture on the right.
\begin{center}
  \begin{picture}(280,130)
    \put(0,65){
      \put(-10,60){\RVCenter{\small $Q:$}}
      \put(0,60){\HVCenter{\small $1$}}
      \put(20,40){\HVCenter{\small $2$}}
      \put(40,20){\HVCenter{\small $3$}}
      \put(60,0){\HVCenter{\small $4$}}      
      \put(40,-20){\HVCenter{\small $5$}}
      \put(20,-40){\HVCenter{\small $6$}}
      \put(0,-60){\HVCenter{\small $7$}}
      \multiput(17,43)(20,-20){3}{\vector(-1,1){14}}
      \multiput(17,-43)(20,20){3}{\vector(-1,-1){14}}
    }
    \put(100,65){
      \put(-10,60){\RVCenter{\small $\Gamma_Q:$}}
      \multiput(0,60)(40,0){3}{\circle*{3}}
      \multiput(0,-60)(40,0){3}{\circle*{3}}
      \multiput(20,40)(40,0){2}{\circle*{3}}
      \multiput(20,-40)(40,0){2}{\circle*{3}}
      \put(40,20){\circle*{3}}
      \put(40,-20){\circle*{3}}
      \put(60,0){\HVCenter{\small $1$}}
      \put(80,20){\HVCenter{\small $2$}}
      \put(80,-20){\HVCenter{\small $3$}}
      \put(100,40){\HVCenter{\small $4$}}      
      \put(100,0){\HVCenter{\small $5$}}
      \put(100,-40){\HVCenter{\small $6$}}
      \put(120,60){\HVCenter{\small $10$}}
      \put(120,20){\HVCenter{\small $7$}}
      \put(120,-20){\HVCenter{\small $8$}}
      \put(120,-60){\HVCenter{\small $11$}}
      \put(140,40){\HVCenter{\small $12$}}      
      \put(140,0){\HVCenter{\small $9$}}
      \put(140,-40){\HVCenter{\small $13$}}
      \put(160,20){\HVCenter{\small $14$}}
      \put(160,-20){\HVCenter{\small $15$}}
      \put(180,0){\HVCenter{\small $16$}}
      \put(3,57){\vector(1,-1){14}}
      \multiput(23,37)(20,20){2}{\vector(1,-1){14}}
      \multiput(43,17)(20,20){3}{\vector(1,-1){13.3}}
      \put(3,-57){\vector(1,1){14}}
      \multiput(23,-37)(20,-20){2}{\vector(1,1){14}}
      \multiput(43,-17)(20,-20){3}{\vector(1,1){13.3}}
      \multiput(23,43)(20,-20){2}{\vector(1,1){14}}
      \multiput(23,-43)(20,20){2}{\vector(1,-1){14}}
      \put(63,43){\vector(1,1){14}}
      \put(63,-43){\vector(1,-1){14}}
      \multiput(0,0)(20,-20){3}{
        \multiput(63.7,3.7)(20,20){2}{\vector(1,1){12.6}}
      }
      \multiput(0,0)(20,20){3}{
        \multiput(63.7,-3.7)(20,-20){2}{\vector(1,-1){12.6}}
      }
      \multiput(103.7,43.7)(20,-20){3}{\vector(1,1){11.8}}
      \multiput(103.7,-43.7)(20,20){3}{\vector(1,-1){11.8}}
      \multiput(124.9,55.1)(20,-20){3}{\vector(1,-1){10.2}}
      \multiput(124.9,-55.1)(20,20){3}{\vector(1,1){10.2}}
    }
  \end{picture}
\end{center}
We choose $M=\oplus_{i=1}^{16}M_i$ with $M_i$ corresponding to the vertex marked
by $i$ in the Auslander-Reiten quiver. Then $E_\oM$ is given by the following
quiver with relations:
\begin{center}
  \begin{picture}(80,90)
    \put(0,45){
      \put(0,0){\HVCenter{\small $1$}}
      \put(20,20){\HVCenter{\small $2$}}
      \put(20,-20){\HVCenter{\small $3$}}
      \put(40,40){\HVCenter{\small $4$}}      
      \put(40,0){\HVCenter{\small $5$}}
      \put(40,-40){\HVCenter{\small $6$}}
      \put(60,20){\HVCenter{\small $7$}}
      \put(60,-20){\HVCenter{\small $8$}}
      \put(80,0){\HVCenter{\small $9$}}
      \multiput(0,0)(20,-20){3}{
        \multiput(3.7,3.7)(20,20){2}{\vector(1,1){12.6}}
      }
      \multiput(0,0)(20,20){3}{
        \multiput(3.7,-3.7)(20,-20){2}{\vector(1,-1){12.6}}
      }
      \multiput(0,0)(20,20){2}{
        \multiput(0,0)(20,-20){2}{
          \qbezier[17](6,0)(20,0)(34,0)
        } 
      }
    }
  \end{picture}
\end{center}
Again use \cite{BdP99} and \cite{GL87} to check that $\lmd{E_\oM}$ is derived equivalent to 
$\coh\XX$ for $\XX$ of weight type $(4,4,2)$. In particular, $E_\oM$ is piecewise hereditary.
A straightforward check shows that  
$\mu_5\mu_2\mu_8\mu_1\mu_9\mu_5\mu_7\mu_3(\hQ_\oM)=\Delta$, where $\Delta$ is the quiver in Figure~\ref{fig:ellipticR}  associated to $\sE_7^{(1,1)}$.

\subsubsection{} The case $\sE_8^{(1,1)}$ has already been studied extensively
in~\cite[Sec.~14-17 and 19.4]{GLS05}. For convenience we recall the
aspects which are here relevant.
Let $Q$ be the quiver as shown in the following figure on the left.
The Auslander-Reiten quiver $\Gamma_Q$ of $Q$ has then the shape as shown in 
the middle of the following picture.
\begin{center}
  \begin{picture}(280,90)
    \put(0,45){
      \put(-10,40){\RVCenter{\small $Q:$}}
      \put(0,40){\HVCenter{\small $1$}}
      \put(0,0){\HVCenter{\small $2$}}
      \put(0,-40){\HVCenter{\small $3$}}
      \put(20,20){\HVCenter{\small $4$}}      
      \put(20,-20){\HVCenter{\small $5$}}
      \multiput(16.3,16.3)(0,-40){2}{\vector(-1,-1){12.6}}
      \multiput(16.3,-16.3)(0,40){2}{\vector(-1,1){12.6}}
      }
    \put(85,45){
      \put(-10,40){\RVCenter{\small $\Gamma_Q:$}}
      \put(0,40){\HVCenter{\small $1$}}
      \put(0,0){\HVCenter{\small $2$}}
      \put(0,-40){\HVCenter{\small $3$}}
      \put(20,20){\HVCenter{\small $4$}}      
      \put(20,-20){\HVCenter{\small $5$}}
      \put(40,40){\HVCenter{\small $6$}}
      \put(40,0){\HVCenter{\small $7$}}
      \put(40,-40){\HVCenter{\small $8$}}
      \put(60,20){\HVCenter{\small $9$}}
      \put(60,-20){\HVCenter{\small $10$}}
      \put(80,40){\HVCenter{\small $11$}}      
      \put(80,0){\HVCenter{\small $12$}}
      \put(80,-40){\HVCenter{\small $13$}}
      \put(100,20){\HVCenter{\small $14$}}
      \put(100,-20){\HVCenter{\small $15$}}
      \multiput(3.7,36.3)(20,-60){2}{\vector(1,-1){12.6}}
      \multiput(3.7,-3.7)(20,20){3}{\vector(1,-1){12.6}}
      \multiput(3.7,-36.3)(20,60){2}{\vector(1,1){12.6}}
      \multiput(3.7,3.7)(20,-20){2}{\vector(1,1){12.6}}
      \put(43.7,3.7){\vector(1,1){12.6}}
      \multiput(43.7,-36.3)(20,60){2}{\vector(1,1){11.8}}
      \multiput(43.7,-3.7)(20,20){2}{\vector(1,-1){11.8}}
      \put(64.9,-15.1){\vector(1,1){10.2}}
      \put(64.9,-24.9){\vector(1,-1){10.2}}
      \multiput(84.9,4.9)(0,-40){2}{\vector(1,1){10.2}}
      \multiput(84.9,-4.9)(0,40){2}{\vector(1,-1){10.2}}
    }
    \put(240,45){
      \put(-10,40){\RVCenter{\small $\tQ_\oM:$}}
      \put(0,40){\HVCenter{\small $1$}}
      \put(0,0){\HVCenter{\small $2$}}
      \put(0,-40){\HVCenter{\small $3$}}
      \put(20,20){\HVCenter{\small $4$}}      
      \put(20,-20){\HVCenter{\small $5$}}
      \put(40,40){\HVCenter{\small $6$}}
      \put(40,0){\HVCenter{\small $7$}}
      \put(40,-40){\HVCenter{\small $8$}}
      \put(60,20){\HVCenter{\small $9$}}
      \put(60,-20){\HVCenter{\small $10$}}
      \multiput(3.7,36.3)(20,-60){2}{\vector(1,-1){12.6}}
      \multiput(3.7,-3.7)(20,20){3}{\vector(1,-1){12.6}}
      \multiput(3.7,-36.3)(20,60){2}{\vector(1,1){12.6}}
      \multiput(3.7,3.7)(20,-20){2}{\vector(1,1){12.6}}
      \put(43.7,3.7){\vector(1,1){12.6}}
      \put(43.7,-36.3){\vector(1,1){11.8}}
      \put(43.7,-3.7){\vector(1,-1){11.8}}
      \multiput(0,-40)(0,40){3}{\qbezier[17](6,0)(20,0)(34,0)}
      \multiput(20,-20)(0,40){2}{\qbezier[17](6,0)(20,0)(34,0)}
    }
  \end{picture}
\end{center}
We choose $M=\oplus_{i=1}^{15}M_i$ the direct sum of all indecomposable modules.
Thus, $E_\oM$ is given by the following 
quiver $\tQ_\oM$ with relations as shown in the previous picture on the right.
Again, \cite{BdP99} and \cite{GL87} can be used to check that 
$\lmd{E_\oM}$ is derived equivalent to $\coh\XX$ where $\XX$ is a weighted 
projective line of weight type $(6,3,2)$ (in fact $E_\oM$ is tubular as can be 
seen using \cite[Sect.~4.7]{Ri84}), in particular it is
piecewise hereditary. A straightforward check shows that  
$\mu_2\mu_5\mu_4\mu_{10}\mu_9\mu_8\mu_3\mu_5\mu_7\mu_5\mu_9\mu_8\mu_3\mu_6\mu_1(\hQ_\oM)=\Delta$, 
where $\Delta$ is the quiver in Figure~\ref{fig:ellipticR}  
associated to $\sE_8^{(1,1)}$. 

This finishes the proof of Theorem~\ref{thm:main1} for the cases
$\sE_6^{(1,1)},  \sE_7^{(1,1)}$ and~$\sE_8^{(1,1)}$.

%%%%%%%%%%%%%%%%%%%%%%%%%%%%%%%%%%%%%%%%%%%%%%%%%%
\subsection{The tubular D-case}
Let us point out first that in this case the method from~\ref{ssec:GLS}
can't work: The derived tubular algebras of type $(2,2,2,2)$ are well-known.
The 9 types are listed for example in~\cite[p.~652]{BdP99} as quivers with 
relations.
A quick check shows that in each case there is either a vertex in the quiver
where 2 relations start resp.~end, or there is a relation of length 3.
Thus, none of these algebras can be of the form $E_\oM$ for any
quiver $Q$ and a terminal $\CC Q$-module $M$. Moreover, we were unable
to find a Hom-finite Frobenius category $\mathcal{F}$ such that the stable
category $\underline{\mathcal{F}}$ is equivalent to the cluster category
$\cC_\XX$ where $\XX$ is a weighted projective line of type $(2,2,2,2)$.

In any case, we can use Palu's cluster character~\cite{Pa08}
\[
X_?\df\cC_\XX\rightarrow \CC[x_1^{\pm 1},\cdots,x_6^{\pm 1}].
\]
Since the exchange graph for cluster tilting objects in $\cC_\XX$ is connected 
\cite[Thm.~8.8]{BKL08},
$X_?$ induces a {\em surjection} from the indecomposable rigid objects in
$\cC_\XX$ (which are in bijection with the positive real Schur roots) and the
cluster variables of the corresponding tubular cluster algebra. 
Thus, it remains to show that in the D-case the cluster variables
$X_M$ are pairwise different for non-isomorphic, indecomposable rigid
objects $M$ in $\cC_\XX$. 
This occupies essentially the rest of the paper, 
see Remark~\ref{rem:main}~(b) for our strategy.
Note, that this already implies that $X_?$ yields a bijection between 
basic cluster tilting objects and clusters.
% This occupies essentially the rest of the paper, 
%%%%%%%%%%%%%%%%%%%%%%%%%%%%%%%%%%%%%%%%%%%%%%%%%%%%
%%%%%%%%%%%%%%%%%%%%%%%%%%%%%%%%%%%%%%%%%%%%%%%%%%%%

\section{Combinatorics of real Schur roots in the case $(2,2,2,2)$}
If not otherwise mentioned, we have in this section 
$\blam=([1:0],[0:1],[1:1],[\rho:1])$ and $\XX=(\PP^1,\blam,(2,2,2,2))$.

\subsection{A coordinate system} 
It is well-known that in this situation there exists a tilting object
$T\in\coh\XX$ such that the Gabriel quiver of $A=\End_\XX(T)$ has the
following shape
\begin{equation}
  \label{eq:quiver1}
  \begin{picture}(60,60)
    \put(0,10){
      \multiput(0,0)(30,0){2}{
        \put(5,5){\vector(1,1){20}}
        \put(5,25){\vector(1,-1){20}}
        \put(5,30){\vector(1,0){20}}
        \put(5,0){\vector(1,0){20}}
        }
      \put(0,30){\HVCenter{\small $1$}}
      \put(0,0){\HVCenter{\small $2$}}
      \put(30,30){\HVCenter{\small $3$}}
      \put(30,0){\HVCenter{\small $4$}}
      \put(60,30){\HVCenter{\small $5$}}
      \put(60,0){\HVCenter{\small $6$}}
      }
  \end{picture},
\end{equation}
see for example~\cite{BdP99}. Then the triangle equivalence
\[
\mathbf{R}\Hom_\XX(T,-)\df\cDb(\coh\XX)\rightarrow\cDb(\lmd{A})
\]
induces an isometry between $(\KX,\euler{-,-})$ and 
$(K_0(\lmd{A}),\euler{-,-}_A)$
where
\[
\euler{\dimv X,\dimv Y}_A=\dim\Hom_A(X,Y)-\dim\Ext^1_A(X,Y)+\dim\Ext^2_A(X,Y)
\]
is the 
homological form for $A$. As usual we will
use in $K_0(\lmd{A})$ the basis given by the simple $A$-modules.  We
will describe the positive real Schur roots and their combinatorics in
this coordinate system.
We denote a vector $\bv\in \ZZ^6$ usually by
$\bv=\dimvec{\bv(1)}{\bv(2)}{\bv(3)}{\bv(4)}{\bv(5)}{\bv(6)}$.

%%%%%%%%%%%%%%%%%%%%%%%%%%%%%%%%%%%%%%%%%%%%%%%%%%%%
\subsection{Description of the basic roots}

In  the following we shall use the quaternion group 
$$
H=\{\pm 1,\pm i,\pm j,\pm k\}
$$
with $ij=k=-ji$, $jk=i=-kj$ and $ki=j=-ik$.
Let $H^+= \{1,i,j,k\}$. 

Let 
$$
\bh_0=\dimvec001111,\quad 
\bh_1=\dimvec112211\quad \text{and}\quad 
\bh_\infty=\dimvec111100.
$$
Observe that $\euler{\bh_0,\bh_\infty}=-\euler{\bh_\infty,\bh_0}=2$.

Furthermore we define vectors $\bv_q^x$ for $q=0,1,\infty$ and $x\in
H^+$ as follows:
\begin{align*}
  \bv_0^1&=\dimvec001010,&
  \bv_1^1&=\dimvec101110,&
  \bv_\infty^1&=\dimvec100100,\\
  \bv_0^i&=\dimvec{-1}00000,&
  \bv_1^i&=\dimvec011110,&
  \bv_\infty^i&=\dimvec111110,\\
  \bv_0^j&=\dimvec001001,&
  \bv_1^j&=\dimvec001000,&
  \bv_\infty^j&=\dimvec00000{-1},\\
  \bv_0^k&=\dimvec011111,&
  \bv_1^k&=\dimvec112111,&
  \bv_\infty^k&=\dimvec101000.
\end{align*}
For $x\in H^+$ and $q=0,1,\infty$ define 
$$
\bv_q^{-x}=\bh_q-\bv_q^x.
$$
The 
homological form can be calculated explicitly by 
$$
\euler{x,y}=x^\transp \text{\small $\left[
    \begin{matrix}
       1& 0&-1&-1& 1& 1\\
       0& 1&-1&-1& 1& 1\\
       0& 0& 1& 0&-1&-1\\
       0& 0& 0& 1&-1&-1\\
       0& 0& 0& 0& 1& 0\\
       0& 0& 0& 0& 0& 1\\
    \end{matrix}\right]$}y.
$$

The following result is crucial for the sequel. % forthcoming argument?

\begin{lemma}
  \label{lem:crucial}
  The following formula hold for all $x, h \in H$.
  \begin{equation}
    \label{eq:formula1}
    \euler{\bv_0^x,\bv_1^{hx}}=
    \euler{\bv_1^x,\bv_\infty^{hx}}=
    \euler{\bv_0^{-hx}, \bv_\infty^{x}}=
      \begin{cases}
        1,&\text{if $h\in H^+$},\\
        0,&\text{if $h\not\in H^+$};
      \end{cases}
  \end{equation}
  and also
  \begin{align}
    \label{eq:formula2}
    \euler{\bv_0^x,\bh_\infty}=\euler{\bv_1^x,\bh_\infty}=\euler{\bh_0,\bv_1^x}=\euler{\bh_0,\bv_\infty^x}=1.
  \end{align}
\end{lemma}

\begin{proof}
  This is an immediate (although lengthy) calculation.
\end{proof}

%%%%%%%%%%%%%%%%%%%%%%%%%%%%%%%%%%%%%%%%%%%%%%%%%%
\subsection{Description of all real Schur roots}

Recall that we defined $\QQi=\mathbb{Q}\cup\{\infty\}$. For
$q\in \QQi$ the numbers $a(q)\in\ZZ$ and $b(q)\in \ZZ_{\geq 0}$ are defined
by the properties
$\gcd(a(q),b(q))=1$ and $q=\frac{a(q)}{b(q)}$.
Furthermore, for $q\in\QQi$ we define its \emph{type} by
$$
\typ{q}=
\begin{cases}
  0,&\text{if $a(q)\equiv 0$, $b(q)\equiv 1 \mod 2$},\\
  1,&\text{if $a(q)\equiv 1$, $b(q)\equiv 1 \mod 2$},\\
  \infty,&\text{if $a(q)\equiv 1$, $b(q)\equiv 0 \mod 2$}.
\end{cases}
$$
For $n\in \ZZ$ let
$$
\innerfunc{n}=
\begin{cases}
  \tfrac{n}{2},&\text{if $n$ is even},\\
  \tfrac{n-1}{2},&\text{if $n$ is odd}.
\end{cases}
$$

For $q\in \QQi$ and $x\in H$ define
\begin{equation}
  \label{eq:def-gen-vec}
  \bv_q^x=\bv_{\typ{q}}^x+\innerfunc{b(q)}\bh_0+\innerfunc{a(q)}\bh_\infty
  \quad\text{and}\quad
  \bh_q=b(q)\bh_0+a(q)\bh_\infty.
\end{equation}

\begin{lemma}
  \label{lem:correct-slope}
  We have $\slope(\bh_q)=\slope(\bv_q^x)=q$ for each $q\in\QQi$ and $x\in H$. In
  particular $\euler{\bv_q^x,\bh_q}=\euler{\bh_q,\bv_q^x}=0$. 
  Furthermore, we have  $\bv_q^x+\bv_q^{-x}=\bh_q$ and $\bv_p^x = \bv_q^y$ implies $p=q$ and $x=y$.
\end{lemma}

\begin{proof}
  $\slope(\bh_q)=q$ follows directly from the definition.
  Moreover, we have $\euler{\bh_0,\bv_q^x}=\euler{\bh_0,\bv_{\typ{q}}^x}+
  \innerfunc{a(q)}\euler{\bh_0,\bh_\infty}=\euler{\bh_0,\bv_{\typ{q}}^x}+2
  \innerfunc{a(q)}=a(q)$, where the last equation can easily be verified in
  each of the three cases $q=0,1,\infty$. Similarly
  $\euler{\bv_q^{x}, \bh_\infty}=b(q)$ and therefore $\slope(\bv_q^x)=q$. 

  If $\typ{q}=0$ then
  $\bv_q^x+\bv_q^{-x}=\bv_0^x+\bv_0^{-x}+(b(q)-1)\bh_0+a(q)\bh_\infty=\bh_q$, the
  latter since $\bv_0^x+\bv_0^{-x}=\bh_0$. Similarly $\bv_q^x+\bv_q^{-x}=\bh_q$ is
  verified in the two remaining cases where $\typ{q}=1,\infty$.
\end{proof}

We denote by $\Exc$ the class of indecomposable rigid objects in $\coh
\XX$.

\begin{proposition}
  \label{prop:exceptional-to-vectors}
The map $E\mapsto \dimv E$ induces a bijection from the isoclasses of 
$\Exc$ to the set $\{\bv_q^x \mid q\in\QQi, x\in H\}$.
\end{proposition}

\begin{proof}
By Proposition~\ref{prp:SchurRt} 
$\{\dimv E\mid E\in \Exc\}$ is precisely the set of the real positive
Schur roots. Moreover, for $E\in\coh(\XX)$ with $\Ext^1_\XX(E,E)$ the
isomorphism class of $E$ is uniquely determined by $\dimv E$.
Since $\XX$ is of type $(2,2,2,2)$ we have  by
Proposition~\ref{prp:SchurRt} and Theorem~\ref{coh:thm-tub} that
$0\prec\be\in\KX$ is a real positive Schur root if and only if
$\euler{\be,\be}=1$ and $\ql(\be)=1$. Moreover, for each $q\in\QQi$ 
there are precisely $8$ positive real Schur roots $\be$ with
$\slope(\be)=q$. By Lemma~\ref{lem:correct-slope} the $\bv_q^x$ with
$x\in H$ are precisely $8$ elements with the required properties.
 \end{proof}

\begin{remark}
\label{rem:related-type}
Note, that for our description of the real Schur roots the adequate choice
of $24 = 4\times (2^2-1)\times 2$ basic roots was essential. This is the number
of roots for $\sD_4$.

In~\cite[Sec.~15]{GLS05} a similar description of the real Schur roots was 
discussed for the tubular case $(6,3,2)$. In that case
$240= (6^2-1)\times 6 + (3^2-1)\times 3 + (2^2-1)\times 2$ have to be chosen
properly. This is the number of roots for $\sE_8$.

For the tubular cases $(4,4,2)$ and $(3,3,3)$ one would need
to find $126=2\times(4^2-1)\times+(2^2-1)\times 2$
respectively $3\times(3^2-1)\times 3=72$ basic roots, that is,
precisely the number of roots in
a root system of type $\sE_7$ respectively $\sE_6$.
\end{remark}

%%%%%%%%%%%%%%%%%%%%%%%%%%%%%%%%%%%%%%%%%%%%%%%%%%
\subsection{Orthogonality}

In the sequel we shall need the following definition of
``distance'' 
between different slopes.

\begin{definition}
For $p,q\in \QQi$ we define
$$
\qdist{p,q}
=\abs{a(q)b(p)-a(p)b(q)}.
$$
\end{definition}

\begin{lemma}
  Let $p,q\in\{0,1,\infty\}$ and $x,y\in H$. Then
  $$
  \euler{\bv_p^x, \bv_q^y}=0
  \Leftrightarrow
  \begin{cases}
    x=y, &\  \text{if
      $\typ{p}=\typ{q}$},\\
    y\in -H^+ x, &\ \text{if
      $(p,q)=(0,1), (1,\infty), (\infty, 0)$},\\
    x\in H^+ y, &\ \text{if
      $(p,q)=(1,0), (\infty,1), (0,\infty)$}.
  \end{cases}
  $$
\end{lemma}

\begin{proof}
  This follows from Lemma \ref{lem:crucial} and the fact that
  $\euler{\bv_p^x,\bv_q^y}=\euler{\Phi \bv_q^y,\bv_p^x}=\euler{\bv_q^{-y},\bv_p^x}$.
\end{proof}

\begin{proposition}
  \label{prop:orto}
  Let $p,q\in \QQi$ and $x,y\in H$. If 
  $p>q$ then 
  $\euler{\bv_p^x,\bv_q^y}=0$ 
  if and only if the condition of the
  corresponding cell in the following table is satisfied.
  \begin{equation}
    \label{eq:cond-comp}
    \begin{tabular}{c|c|c|c|}
      & $\typ{q}=0$ & $\typ{q}=1$ & $\typ{q}=\infty$\\
      \hline
      $\typ{p}=0$ & 
        $\begin{matrix} x=y\\ \qdist{p,q}=2\end{matrix}$ & 
        $\begin{matrix} y\in H^+ x\\ \qdist{p,q}=1\end{matrix}$ & 
        $\begin{matrix} x\in -H^+y\\ \qdist{p,q}=1\end{matrix}$ \\
      \hline
      $\typ{p}=1$ & 
        $\begin{matrix} x\in H^+y\\ \qdist{p,q}=1\end{matrix}$ &
        $\begin{matrix} x=y\\ \qdist{p,q}=2\end{matrix}$ & 
        $\begin{matrix} y\in H^+ x\\ \qdist{p,q}=1\end{matrix}$ \\ 
      \hline
      $\typ{p}=\infty$ & 
        $\begin{matrix} y\in -H^+ x\\ \qdist{p,q}=1\end{matrix}$ &
        $\begin{matrix} x\in H^+y\\ \qdist{p,q}=1\end{matrix}$ &
        $\begin{matrix} x=y\\ \qdist{p,q}=2\end{matrix}$ \\
      \hline
    \end{tabular}
  \end{equation}
\end{proposition}

\begin{proof}
  First suppose $\typ{p}=\typ{q}=0$. To calculate
  $\euler{\bv_p^x,\bv_q^y}$ use Definition~\ref{eq:def-gen-vec},
  bilinearity and then Lemma
  \ref{lem:correct-slope} and equations \eqref{eq:formula2} of Lemma
  \ref{lem:crucial} in order to obtain 
  $$
  \euler{\bv_p^x,\bv_q^y}=\euler{\bv_0^x,\bv_0^y}+\left(2\innerfunc{b(q)}+1\right)\innerfunc{a(p)}-
  \left(2\innerfunc{b(p)}+1\right)\innerfunc{a(q)}.
  $$
  Observe that $\left(2\innerfunc{b(p)}+1\right)=b(p)$, $\left(2\innerfunc{b(q)}+1\right)=b(q)$ and
  $\innerfunc{a(p)}=\frac{a(p)}{2}$, $\innerfunc{a(q)}=\frac{a(q)}{2}$. Hence 
  $$
  \euler{\bv_p^x,\bv_q^y}=\euler{\bv_0^x,\bv_0^y}-\tfrac{a(p)b(q)-a(q)b(p)}{2}=
  \euler{\bv_0^x,\bv_0^y}-\tfrac{\qdist{p,q}}{2},
  $$ 
  the latter since $p>q$, which implies that
  $a(p)b(q)-a(q)b(p)>0$. Therefore $\euler{\bv_0^x,\bv_0^y}>0$. This in
  turn implies $\euler{\bv_0^x,\bv_0^y}=1$ thus $x=y$ and
  $\qdist{p,q}=2$. This shows the result in that case. The other two
  cases where $\typ{p}=\typ{q}$ are quite similar.
  
  Similarly all other cases are calculated. For instance if $\typ{p}=0$
  and $\typ{q}=1$ then we calculate using similar arguments as above
  that
  $\euler{\bv_p^x,\bv_q^y}=\euler{\bv_0^x,\bv_1^y}+\tfrac{-\qdist{p,q}-1}{2}$.
  Since the first summand is at most $1$ and the second at most $-1$ we
  get that $\euler{\bv_p^x,\bv_q^y}=0$ if and only if
  $\euler{\bv_0^x,\bv_1^y}=1$ and $\qdist{p,q}=-1$, which happens if and
  only if $y\in H^+ x$ and $\qdist{p,q}=1$.
\end{proof}

\begin{corollary}
  \label{cor:orto}
  Let $p,q\in \QQi$ and $x,y\in H$. If 
  $p<q$ then $\euler{\bv_p^x, \bv_q^y}=0$ if and only if
  the condition of the
  corresponding cell in the following table is satisfied.
  \begin{center}
    \begin{tabular}{c|c|c|c|}
      & $\typ{q}=0$ & $\typ{q}=1$ & $\typ{q}=\infty$\\
      \hline
      $\typ{p}=0$ & 
        $\begin{matrix} x=-y\\ \qdist{p,q}=2\end{matrix}$ & 
        $\begin{matrix} y\in -H^+ x\\ \qdist{p,q}=1\end{matrix}$ & 
        $\begin{matrix} x\in H^+y\\ \qdist{p,q}=1\end{matrix}$ \\
      \hline
      $\typ{p}=1$ & 
        $\begin{matrix} x\in -H^+y\\ \qdist{p,q}=1\end{matrix}$ &
        $\begin{matrix} x=-y\\ \qdist{p,q}=2\end{matrix}$ & 
        $\begin{matrix} y\in -H^+ x\\ \qdist{p,q}=1\end{matrix}$ \\ 
      \hline
      $\typ{p}=\infty$ & 
        $\begin{matrix} y\in H^+ x\\ \qdist{p,q}=1\end{matrix}$ &
        $\begin{matrix} x\in -H^+y\\ \qdist{p,q}=1\end{matrix}$ &
        $\begin{matrix} x= -y\\ \qdist{p,q}=2\end{matrix}$ \\
      \hline
    \end{tabular}
  \end{center}
\end{corollary}

\begin{proof}
  Observe that $\euler{\bv_p^x, \bv_q^y}=-\euler{\Phi
  \bv_q^y,\bv_p^x}=-\euler{\bv_q^{-y}, \bv_p^x}$ and therefore the result
  follows from the previous proposition.
\end{proof}

\begin{lemma}\label{lem:orto}
  We have
  $\euler{\bv_p^x,\bv_p^y}=0$ if and only if $x\neq - y$.
\end{lemma}

\begin{proof}
  This follows by direct calculations for $p=0,1,\infty$ and in the
  general case using \eqref{eq:def-gen-vec} from
  $\euler{\bv_p^x,\bv_p^y}=\euler{\bv_{\typ{p}}^x,\bv_{\typ{p}}^y}$ since the 
  other summands cancel each other.
\end{proof}

\begin{proposition}
  Let $p,q\in \QQi$ and $x,y\in H$. If 
  $\euler{\bv_p^x, \bv_q^y}=0=\euler{\bv_q^y, \bv_p^x}$ then
  $p=q$ and $x\neq \pm y$.
\end{proposition}

\begin{proof}
  It follows from Proposition \ref{prop:orto} and Corollary
  \ref{cor:orto} that $p=q$. Then the result follows from the previous lemma.
\end{proof}

%%%%%%%%%%%%%%%%%%%%%%%%%%%%%%%%%%%%%%%%%%%%%%%%%%
%%%%%%%%%%%%%%%%%%%%%%%%%%%%%%%%%%%%%%%%%%%%%%%%%%
\section{Triangulations of the sphere with 4 punctures}

%%%%%%%%%%%%%%%%%%%%%%%%%%%%%%%%%%%%%%%%%%%%%%%%%%
\subsection{Definition of triangulated surfaces}
\label{subsec:def-surf-marked}
In \cite{FST}, Fomin, Shapiro and Thurston established a connection
between cluster algebras and triangulated surfaces. We will use
this approach for tubular cluster algebras of type $(2,2,2,2)$.

Let $\surf$ be an oriented 2-dimensional Riemann surface (possibly
with boundary) and $\marked$ a finite set of points in the closure of
$\surf$.  An \emph{arc} in $(\surf,\marked)$ is (the homotopy class in
$\surf\setminus\marked$ of) a curve without self-intersections
connecting two marked points which is not contractible in
$\surf\setminus\marked$ nor deformable into the boundary of
$\surf$. Also a curve and its inverse will be identified. 
Two arcs are \emph{compatible} if they contain curves which
do not intersect in $\surf\setminus\marked$. Each arc is compatible
with itself. A maximal collection of pairwise compatible curves is
called an \emph{ideal triangulation} of $(\surf,\marked)$.

Some of these arcs (namely those which are not loops enclosing a
  single marked point) are then enhanced to \emph{tagged arcs}, that
is, to each of its two endpoints one of the two labels ``plain'' or
``notched'', is attached, see \cite{FST} for details. Tagged arcs
should be seen as a generalization of ``plain'' arcs. If $\beta$
is a tagged arc we denote by $\beta^\circ$ the ``underlying''
untagged arc obtained from $\beta$ by removing the labels on the endpoints.  
Recall from \cite[Definition 7.4]{FST} that if
$\beta^\circ\neq\gamma^\circ$ then $\beta$ and $\gamma$ are
compatible if and only if the tags on common endpoints are the same
and $\beta^\circ$ and $\gamma^\circ$ are compatible; and if
$\beta^\circ=\gamma^\circ$ then then $\beta$ and $\gamma$ are
compatible if and only if they share the same tag on at least one
endpoint.
A \emph{tagged triangulation} is, by definition, a maximal
collection of pairwise compatible tagged arcs. The advantage of the
tags is that for each tagged triangulation $T$ and each arc $\gamma$
there exists a unique tagged arc $\gamma'\neq \gamma$ such that
$\mu_\gamma(T)=T\setminus\{\gamma\}\cup\{\gamma'\}$ is again a tagged 
triangulation, called the \emph{flip of $T$ along $\gamma$}.

To each tagged triangulation $T$ of a marked surface $(\surf,\marked)$
an integer square matrix $B(T)$, indexed by the arcs of $T$, is
assigned in such a way that the matrix mutation $\mu_k$, as defined by
Fomin-Zelevinsky in \cite{FZ1}, corresponds to the \emph{flip} of
tagged triangulation, that is, $B(\mu_\gamma(T))=\mu_\gamma(B(T))$.
By \cite{FST}, the set of matrices $\mSet(\surf,\marked)$ of matrices
$B(T)$ as $T$ varies through the tagged triangulation of
$(\surf,\marked)$, is a single mutation class of matrices.

Denote by $\tAC(\surf,\marked)$ the \emph{tagged arc complex}, that
is, the clique complex on the set of tagged arcs given by the
compatibility relation. The \emph{tagged exchange graph}
$\tEG(\surf,\marked)$ is by definition the dual
graph of $\tAC(\surf,\marked)$. We quote part of the main result,
Theorem 7.11, of \cite{FST}.

\begin{theorem}[Fomin, Shapiro, Thurston]
  If $\mathcal{A}$ is a cluster algebra whose set of exchange matrices
  is $\mSet(\surf,\marked)$ for some marked surface $(\surf,\marked)$
  then the cluster complex of $\mathcal{A}$ is isomorphic to
  $\tAC(\surf,\marked)$ and the exchange graph of $\mathcal{A}$ is
  isomorphic to $\tEG(\surf,\marked)$. 
\end{theorem}

As a consequence, in the situation of the previous result, the cluster
variables of $\mathcal{A}$ are in bijection with the tagged arcs of
$(\surf,\marked)$. 

\begin{example}
  Let $(\surf,\marked)$ be the 2-sphere with 4 punctures called
  $L,U,D$ and $O$. We visualize the situation as a disk with 3 punctures
  $L,U,D$, and represent the fourth puncture $O$ by its boundary. Let
  $T$ be the tagged triangulation (all tags are plain) showed in the
  picture below. The matrix $B(T)$ is shown on the right hand
  side. Notice that the associated quiver is exactly the one given in
  \eqref{eq:quiver1}.
  
  \begin{center}
    \includegraphics[scale=1,viewport=140 585 470 710,clip]{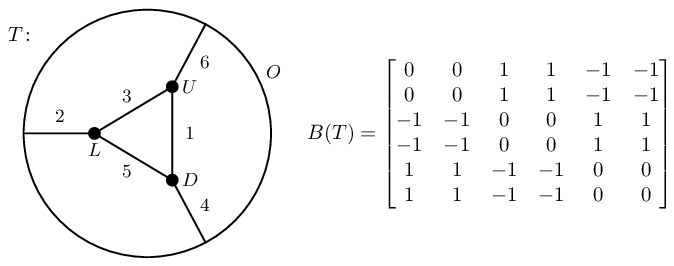}
  \end{center}
\end{example}

%%%%%%%%%%%%%%%%%%%%%%%%%%%%%%%%%%%%%%%%%%%%%%%%%%
\subsection{Special untagged arcs}

We now want to associate with each rational number $p$ two untagged
arcs $\alpha_{q}^\Inner$ and $\alpha_{q}^\Outer$, namely an \emph{inner} arc
$\alpha_{q}^{\Inner}$, connecting two of the three vertices $L,U,D$, and
an \emph{outer} arc $\alpha_{q}^{\Outer}$, connecting one of the
vertices $L,U,D$ with $O$.
We start by doing so for
${q}=-1,0$ and $\infty$.

\begin{definition}
\label{def:low-compl-arcs}
We define $\alpha_q^\Inner$ and $\alpha_q^\Outer$ for $q=-1,0,\infty$
as shown in Figure~\ref{fig:sepc-arcs}.
\end{definition}

\begin{figure}[!h]
  \begin{center}
    \includegraphics[scale=1,viewport=130 598 257 717,clip]{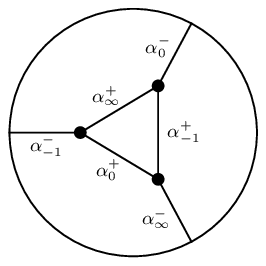}
  \end{center}
\caption{Untagged arcs of low complexity}
\label{fig:sepc-arcs}
\end{figure}

To simplify notations we redefine here the \emph{type} of $q\in\QQi$ as follows:
$$
\typc{q}=
\begin{cases}
  0,&\text{if $a(q)\equiv 0$, $b(q)\equiv 1 \mod 2$},\\
  -1,&\text{if $a(q)\equiv 1$, $b(q)\equiv 1 \mod 2$},\\
  \infty,&\text{if $a(q)\equiv 1$, $b(q)\equiv 0 \mod 2$}.
\end{cases}
$$Notice that $\typc{q}=-\typ{q}$.  

We start by describing first the arc $\alpha_p^{\Inner}$ in the case
where $p=\frac{r}{s}$ with $r,s>0$ are coprime. In the first step we
draw $\innerfunc{r}$ different semicircles with center $U$ and $\innerfunc{s}$
different semicircles with center $D$, all of them open to the left and
mutually non-intersecting. We also draw $\innerfunckl{r+s}$ different
semicircles with center $L$ open to the right. 

\begin{figure}[!h]
  \begin{center}
    \includegraphics[scale=1,viewport=130 557 475 717,clip]{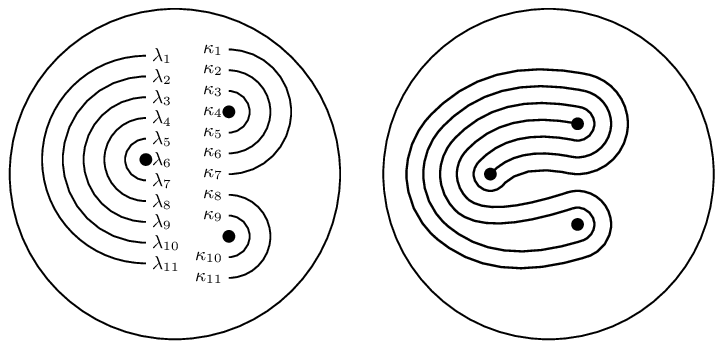}
  \end{center}
\caption{Construction of $\alpha_p^{\Inner}$ for $p=\frac{7}{4}$}
\label{fig:def-arc}
\end{figure}

The endpoints of the semicircles around $U$ and $D$ together with $U$
(resp. $D$) in case $r$ (resp. $s$) is odd, define $r+s$ points
$\kappa_1,\ldots,\kappa_{r+s}$, enumerated consecutively from the top
to the bottom. Similarly, the endpoints of the semicircles around $L$
together with $L$ in case $r+s$ is odd, define $r+s$ endpoints
$\lambda_1,\ldots,\lambda_{r+s}$, enumerated consecutively from the
top to the bottom. In the second step, for each $i$, the point
$\kappa_i$ is joined with $\lambda_i$ by a segment. 
We illustrate this construction in Figure~\ref{fig:def-arc}.

We will show in section \ref{subsec:unfolding} below that each such arc is 
connected.

To define the outer arcs, it will be convenient to use the abbreviation $\outerfunc{n}=\innerfunc{(n-1)}$ for positive $n$, that is,
  $$
  \outerfunc{n}=\begin{cases}
  0,&\text{ if $n=0$,}\\
  \tfrac{n-2}{2},&\text{ if $n>0$ is even,}\\
  \tfrac{n-1}{2},&\text{ if $n$ is odd.}
  \end{cases}
  $$  

Notice that for each arc connecting two different marked points there
exists a unique compatible non-tagged arc connecting the other two vertices. In
this way $\alpha_p^{\Inner}$ defines $\alpha_p^{\Outer}$ in case
$p=\frac{r}{s}$ with $r,s>0$. However, we shall give a more explicit
construction, again in two steps. In the first step we draw
$\outerfunc{r}$ (resp. $\outerfunc{s}$) semicircles around $U$
(resp. $D$). Their endpoints define (possibly together with $U$, resp
$D$) $r-1$ (resp. $s-1$) points denoted by
$\kappa_1,\ldots,\kappa_{r-1}$
(resp. $\kappa_{r+1},\ldots,\kappa_{r+s-1}$). An additional point
$\kappa_r$ is introduced vertically between $\kappa_{r-1}$ and
$\kappa_{r+1}$. Also $\outerfunckl{r+s}$ semicircles are drawn around $L$ with
endpoints $\lam_1,\ldots,\lam_{r+s-1}$. Then $\lam_i$ is joined with
$\kappa_i$ for all $i$.

Additionally $\kappa_r$ is joined with $O$ by a horizontal segment with left 
endpoint $\kappa_r$. The resulting picture is the arc $\alpha_p^{\Outer}$. 
We illustrate this construction in Figure~\ref{fig:outer-arc}.

\begin{figure}[!h]
  \begin{center}
    \includegraphics[scale=1,viewport=130 557 475 717,clip]{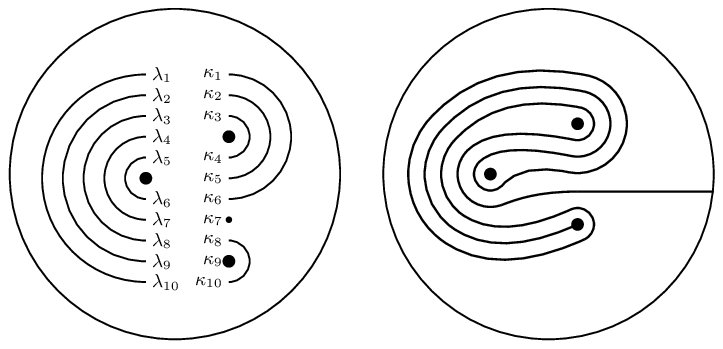}
  \end{center}
\caption{Construction of $\alpha_p^{\Outer}$ for $p=\frac{7}{4}$}
\label{fig:outer-arc}
\end{figure}
 
To define $\alpha_p^{\Inner}$ and $\alpha_p^{\Outer}$ for negative $p$ we 
generalize the above construction. In general, to construct $\alpha_p^{\Inner}$
let $u=\innerfunc{\abs{r}}$, $d=\innerfunc{\abs{s}}$ and 
$l=\innerfunc{\abs{r+s}}$. 
Take the two smallest values among $u,d,l$ and draw as many semicircles around 
the two corresponding marked points. These are then joined around the third 
marked point. Similarly the construction of $\alpha_p^{\Outer}$ is generalized, see Figure~\ref{fig:other-inner}.

\begin{figure}[!h]
  \begin{center}
    \includegraphics[scale=1,viewport=130 638 475 717,clip]{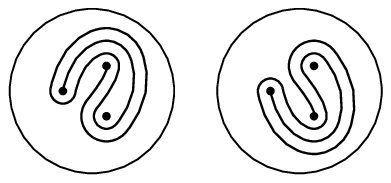}
  \end{center}
\caption{$\alpha_p^{\Inner}$ for $p=\frac{-7}{4}$ (left) and $p=\frac{-4}{7}$ (right)}
\label{fig:other-inner}
\end{figure}

Figure \ref{fig:bigpicture}  
shows how all
inner arcs can be displayed nicely in the plane.

\begin{figure}[!h]
\begin{picture}(300,420)
\put(-60,0){
\includegraphics[scale=.93,viewport=70 300 540 740,clip]{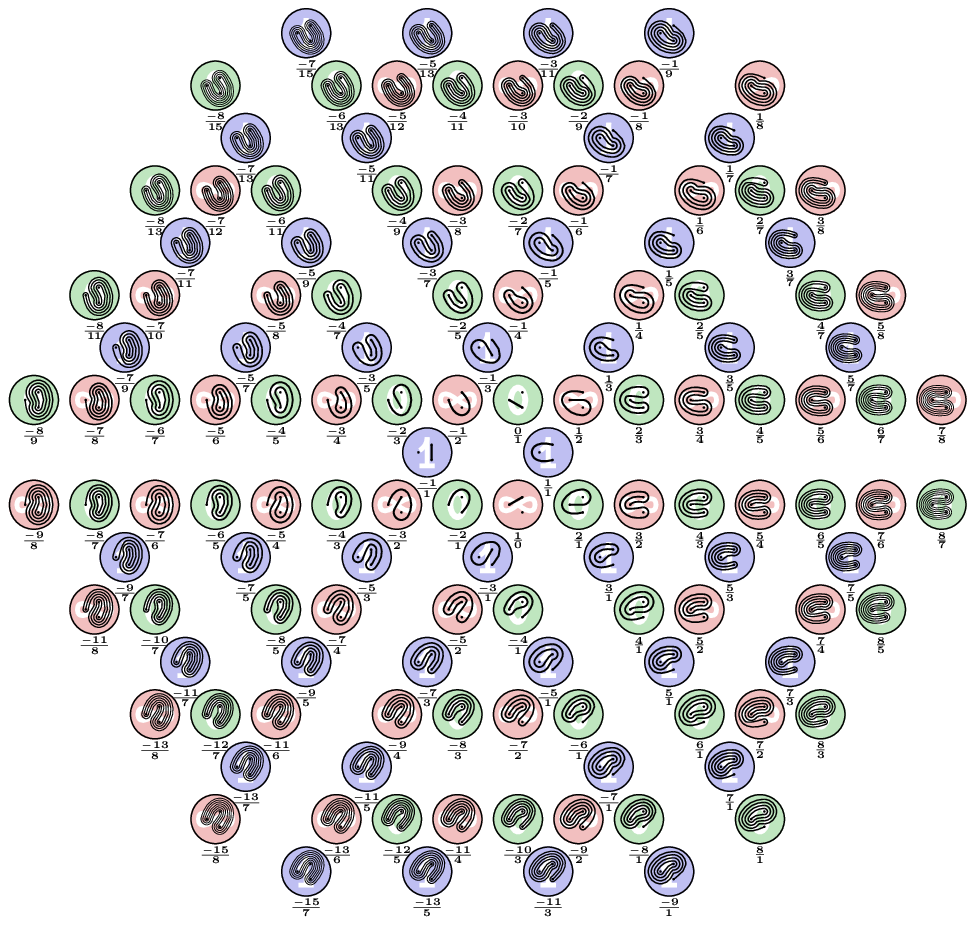}
}
\end{picture}
\caption{Display of inner arcs $\alpha_p^{\Inner}$ for $p\in\QQi$}
\label{fig:bigpicture}
\end{figure}

%%%%%%%%%%%%%%%%%%%%%%%%%%%%%%%%%%
\subsection{Unfolding of arcs}
\label{subsec:unfolding}

We now describe three procedures $\unfold{D}$, $\unfold{U}$ and
$\unfold{L}$ to modify an arc. They will be called \emph{unfolding
  maps}.  For this, let $D'$ (resp. $U'$, $L'$) be the reflection of
$D$ (resp. $U$, $L$) on the line $LU$ (resp. $LD$, $UD$).

\begin{figure}[!h] \label{fig:unfoldL}
  \begin{center}
    \includegraphics[scale=1,viewport=130 615 475 717,clip]{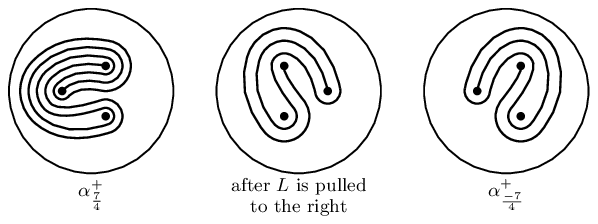}
  \end{center}
\caption{Effect of $\unfold{L}$ on the arc $\alpha_{\frac{7}{4}}^\Inner$.}
\end{figure}

The procedure $\unfold{L}$ consists in pulling the point $L$ 
together with the semicircles
surrounding it between the other two points and place it at the position
$L'$ and then reflect the whole situation on the line $UD$ such that 
$L'$ is reflected back to the position $L$. See Figure~\ref{fig:unfoldL}
for an example.

The other two procedures $\mu_U$ and $\mu_D$ are defined similarly.

\begin{definition}
For each $\frac{r}{s}\in \QQi$ we define its \emph{complexity} by
$$
\complexity{\tfrac{r}{s}}=\abs{r}+\abs{s}+\abs{r+s}.
$$ 
We shall also say that the arc $\alpha_p^\varepsilon$ \emph{has
  complexity} $\complexity{p}$. The lowest complexity is $2$ and occurs if and only if $\frac{r}{s}\in\{-1,0,\infty\}$. 
\end{definition}
\pagebreak[3]

\begin{lemma}
  Let $\alpha_p^{\varepsilon}$ be an arc. 
  \begin{itemize}
  \item[{\rm (a)}] We have
    $\unfold{L}(\alpha_p^\varepsilon)=\alpha_{p'}^\varepsilon$ where
    $p'=-p$. Suppose that $0\leq p$. Then, 
    $\complexity{p}\leq 2$ implies $\complexity{p'}\leq 2$ and if
    $\complexity{p}>2$ then $\complexity{p'}<\complexity{p}$.
  \item[{\rm (b)}] We have
    $\unfold{D}(\alpha_p^\varepsilon)=\alpha_{p'}^\varepsilon$ where
    $p'=\frac{-p}{2p+1}$. Suppose that $-1\leq p\leq 0$. Then, 
    $\complexity{p}\leq 2$ implies $\complexity{p'}\leq 2$ and if
    $\complexity{p}>2$ then $\complexity{p'}<\complexity{p}$.
  \item[{\rm (c)}] We have
    $\unfold{U}(\alpha_p^\varepsilon)=\alpha_{p'}^\varepsilon$ where
    $p'=2-p$. Suppose that $p\leq -2$. Then, 
    $\complexity{p}\leq 2$ implies $\complexity{p'}\leq 2$ and if
    $\complexity{p}>2$ then $\complexity{p'}<\complexity{p}$.
  \end{itemize}
  Hence, for each arc $\alpha_p^\varepsilon$ with $\complexity{p}>2$
  there exists an unfolding map $\unfold{}$ such that
  $\unfold{}(\alpha_p^\varepsilon)=\alpha_{p'}^\varepsilon$ with
  $\complexity{p'}<\complexity{p}$.
\end{lemma}

\begin{proof}
  We shall give the arguments only for the case (a) since 
  it is completely similar for (b) and (c). 
  We start by considering first an
  inner arc $\alpha_p^\Inner$ for some $p=\frac{r}{s}\geq 0$.  Notice
  that there are $m=\min(r,s)$ semicircles around $L$ which connect
  the points $\kappa_1,\ldots,\kappa_m$ above $U$ with the points
  $\kappa_{r+s-m+1},\ldots,\kappa_{r+s}$ below $D$, 
  see Figure~\ref{fig:red-cpxty}.

  After pulling $L$ between $U$ and $D$ to the position $L'$ only
  $\max(r,s)-\min(r,s)$ semicircles remain around $L'$, whereas the
  number of semicircles around $U$ and $D$ remains unchanged. Hence
  the resulting situation is obtained by our construction
  $\alpha_{p'}^\Inner$ where $p'=\frac{-r}{s}$. It is also clear that
  the construction is involutive, hence
  $\unfold{L}(\alpha_p^\Inner)=\alpha_{p'}^\Inner$ even if $p<0$. The
  argument is completely similar for outer arcs, that is
  $\unfold{L}(\alpha_p^\Outer)=\alpha_{-p}^\Outer$.
  
  Observe that $\complexity{p}= 2$ implies for $p\geq 0$ that $p=0$
  or $p=\infty$. Hence $\unfold{L}(\alpha_p^\varphi)=\alpha_p^\varphi$
  for such values.  Now assume that $p=\frac{r}{s}>0$ and
  $\unfold{L}(\alpha_p^\varepsilon)=\alpha_{p'}^\varepsilon$. Then we
  have by construction
  $\complexity{p'}=\abs{r}+\abs{s}+\abs{s-r}<\abs{r}+\abs{s}+\abs{s+r}=\complexity{p}$.
\end{proof}

\begin{figure}[!h]
  \begin{center}
    \includegraphics[scale=1,viewport=185 623 440 715,clip]{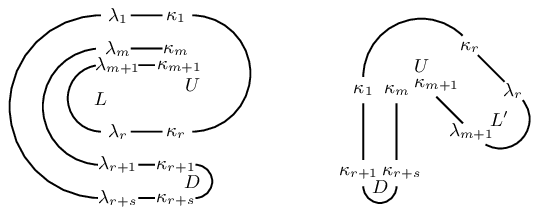}
  \end{center}
\caption{Reducing complexity (assume $r>s$)}
\label{fig:red-cpxty}
\end{figure}

\begin{proposition}
\label{prop:alpha-connected}
For each $p\in\QQi$ the arc $\alpha_p^\varepsilon$ 
is a connected curve which connects the same
endpoints than $\alpha_{\typc{p}}^{\varepsilon}$ as defined in
Definition \eqref{def:low-compl-arcs}.
\end{proposition}
 
\begin{proof}
  The proof is done by induction on the complexity of $p$. Clearly, if
  $\complexity{p}= 2$ then $p\in\{-1,0,\infty\}$ and the arc
  $\alpha_p^{\varepsilon}$ is connected. Otherwise we apply the
  unfolding map $\unfold{L}$, $\unfold{D}$ or $\unfold{U}$ depending whether 
  $0< p$, $-1<p<0$ or $p<-1$ respectively. By the preceding
  Lemma the resulting arc has lower complexity and is therefore
  connected by induction. But it is clear from the definition of the
  unfolding maps that they do not change the connectivity of the
  arcs nor the endpoints. Hence the result.
\end{proof}

%%%%%%%%%%%%%%%%%%%%%%%%%%%%%%%%%%
\subsection{Compatibility of untagged arcs}

We now study when two arcs $\alpha_p^\varepsilon$ and
$\alpha_q^\varphi$ are compatible. For this we start with the simple
case when $q=-1,0,\infty$.

\begin{lemma}
\label{lem:intersection-nr}
  The untagged arc $\alpha_p^\varepsilon$ intersects the untagged arc
  $\alpha_q^\varphi$ (for $q=-1,0,\infty$) precisely
  \begin{equation}
    \label{eq:intersection-nr}
    (\alpha_p^\varepsilon\mid\alpha_q^\varphi)=
    \iofunc{-\varphi\varepsilon}{\qdist{p,q}}
  \end{equation}
  times.  In particular the function $(\alpha_p^\varepsilon \mid -)$ on arcs of low complexity is shown by
  writing the values on $\alpha_q^\varphi$ close to these arcs in the
  following pictures. For this let $p=\frac{r}{s}$ with $r,s$ coprime
  and $s>0$.  
  \begin{center}
    \includegraphics[scale=1,viewport=130 565 470 718,clip]{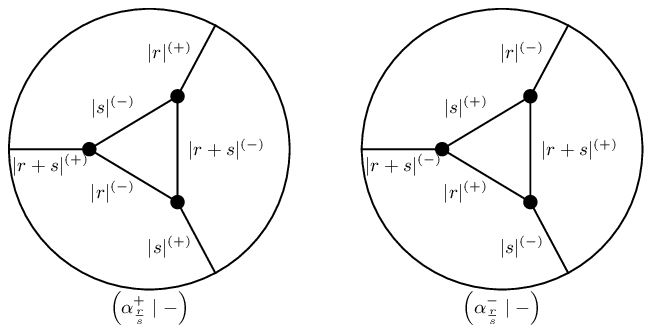}
  \end{center}
\end{lemma}

\begin{proof}
  Let $\beta=\alpha_p^\Inner$ and assume that $p>0$. Then, clearly
  $(\beta\mid \alpha_0^\Outer)$ equals the number of semicircles drawn around the
  point $U$, that is $(\beta\mid\alpha_0^\Outer)= \innerfunc{r}=
  \innerfunc{\qdist{p,0}}$.  Similarly
  $(\beta\mid\alpha_\infty^\Outer)= \innerfunc{s}=
  \innerfunc{\qdist{p,\infty}}$ and since there are
  $\innerfunckl{r+s}$ semicircles around $L$ we have also
  $(\beta\mid\alpha_{-1}^\Outer)= \innerfunckl{r+s}=
  \innerfunc{\qdist{p,-1}}$.  For the following, we use the
  notations as introduced in Figure~\ref{fig:def-arc}.  The segments
  $\kappa_i\lambda_i$ intersect $LU$ if and only if $\lambda_i$ lies
  above $L$ and $\kappa_i$ lies below $U$, that is, if and only if
  $i\leq \innerfunckl{r+s}$ and $i\geq r-\innerfunc{r}$. Therefore, we
  have $(\beta\mid
  \alpha_\infty^\Inner)=\innerfunckl{r+s}-r+\innerfunc{r}\geq 0$ if
  such an index exists and $(\beta\mid \alpha_\infty^\Outer)=0$ otherwise,
  when $\innerfunckl{r+s}-r+\innerfunc{r}<0$ . Thus, in case
  $\typc{p}=0$, we have
  $$
  (\beta\mid \alpha_{\infty}^\Inner)=
  \max(0,\tfrac{r+s-1}{2}-r+\tfrac{r-1}{2})=\max(0, \tfrac{s}{2}-1)=
  \outerfunc{s}=
  \outerfunc{\qdist{p,\infty}}
  $$
  Similarly, in case $\typc{p}=-1$ or $\typc{p}=\infty$, we have
  $$ 
  (\beta\mid \alpha_{\infty}^\Inner)=
  \max(0,\tfrac{r+s}{2}-r+\tfrac{r-1}{2})=\max(0,\tfrac{s-1}{2}-1)=
  \outerfunc{s}=
  \outerfunc{\qdist{p,\infty}}.
  $$ 
  Similarly one verifies that
  $(\beta\mid\alpha_0^\Inner)=\outerfunc{\qdist{p,0}}$.
  This shows that $(\alpha_p^\Inner\mid\alpha_q^\varphi)=\iofunc{-\varphi}{\qdist{p,q}}$
  whenever $p>0$.
  The cases where $\beta=\alpha_p^\Inner$ with $p<0$ are dealt with similarly. Also, the case where
  $\beta=\alpha_p^\Outer$ can be solved in much the same way. This shows
  the claim.
\end{proof}

\begin{proposition}
  \label{prop:comp-untagged}
  Two untagged arcs $\alpha_p^\varepsilon$ and $\alpha_q^\varphi$ are 
  compatible if and only if 
  $$
  \qdist{p,q}\leq \begin{cases}2,&\text{ if $\varepsilon=\varphi$,}\\
    1,&\text{ if $\varepsilon\neq\varphi$.}
    \end{cases}
  $$
\end{proposition}

\begin{proof}
  Let $p=\frac{r}{s}$.  The proof is done by induction on the
  complexity $\complexity{q}$. Assume first that $\complexity{q}=2$,
  that is, $q=-1,0$ or $\infty$ and let  $x:=\qdist{p,q}=\abs{r+s}$, $\abs{r}$ or $\abs{s}$
  respectively.

  By definition, $\alpha_p^\varepsilon$ is compatible with
  $\alpha_q^\varphi$ if and only if
  $(\alpha_p^\varepsilon\mid\alpha_q^\varphi)=0$, that is by Lemma
  \ref{lem:intersection-nr}, if and only if
  $\iofunc{-\varphi\varepsilon}{x}=0$.

  In case $\varphi=\varepsilon$, this is equivalent to
  $\outerfunc{x}=0$, which happens if and only if $x=0$ or $x>0$ is
  even and $x\leq 2$ or $x$ is odd and $x\leq 1$, that is, if and only
  if $x\leq 2$.  Similarly, if $\varphi\neq \varepsilon$, then the two
  arcs are compatible if and only if $\innerfunc{x}=0$, that is, if
  and only if $x\leq 1$.

  If $\complexity{q}>2$ then we apply an appropriate unfolding map to
  both arcs to get $\alpha_{p'}^\varepsilon$ and $\alpha_{q'}^\varphi$
  with $\complexity{q'}<\complexity{q}$ and the result follows since
  the unfolding does not change the number of intersections.
\end{proof}

%%%%%%%%%%%%%%%%%%%%%%%%%%%%%%%%%%%%%%%%%%%%
\subsection{Tagged arcs}

We are now able to define a tagged arc $\alpha_{p,x}$ for each
$p\in\QQi$ and each $x\in H$. First we define such arcs
for slopes of
low complexity, that is for $p=-1,0,\infty$ as shown in 
Figure~\ref{fig:LowCpx}.

\begin{figure}[th] 
  \begin{center}
    \includegraphics[scale=1,viewport=130 596 470 718,clip]{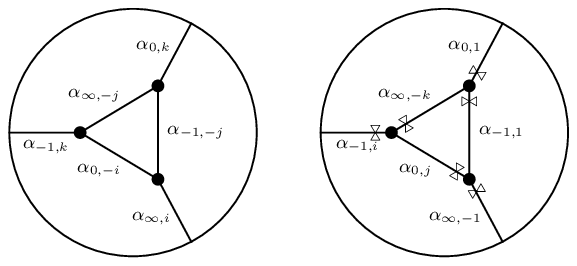}
  \end{center}
\caption{Tagged arcs of low complexity} \label{fig:LowCpx}
\end{figure}
Furthermore, for $p=0,\ -1,\ \infty$, we define $\alpha_{p,-x}$ to be
the arc obtained from $\alpha_{p,x}$ by switching the tags on both
ends from plain to notched and vice versa. In this way we have defined
24 arcs $\alpha_{p,x}$ for $p=0,\ -1,\ \infty$ and $x\in H$. We shall
identify the untagged arcs for these slopes in the obvious way, that
is, $\alpha_{-1}^{\Inner}=\alpha_{-1,-j}$ and
$\alpha_{\infty}^{\Outer}=\alpha_{\infty,i}$, for instance.

\begin{definition} \label{def:tarc}
For each $p\in\QQi$ and 
$x\in H$  define the arc $\alpha_{p,x}$ to be the arc connecting the same vertices as $\alpha_{\typc{p},x}$ 
having the same tags at the endpoints as $\alpha_{\typc{p},x}$, and the untagged 
version of $\alpha_{p,x}$ is either $\alpha_p^{\Inner}$ or $\alpha_p^{\Outer}$. More precisely, 
the untagged version $\alpha_{p,x}^\circ$ of $\alpha_{p,x}$ is as follows:
\begin{equation*}
\alpha_{p,x}^\circ=
\begin{cases}
\alpha_p^{\Inner},&\text{if }(\typc{p},x)\in\{(0,\pm i),(0,\pm j),(-1,\pm 1), (-1,\pm j),(\infty,\pm j),(\infty,\pm k)\},\\
\alpha_p^{\Outer},&\text{else.} 
\end{cases}
\end{equation*}
Note, that two tagged arcs $\alpha_{p,x}$ and $\alpha_{q,y}$ are
isotopic if and only if $(p,x)=(q,y)$.
\end{definition}

\begin{theorem}
  \label{thm:comp}
   Let $p,q\in\QQi$ and $x,y\in H$. Then the two arcs $\alpha_{p,x}$
   and $\alpha_{q,y}$ are compatible if and only if the vectors
   $\bv_p^x$ and $\bv_q^y$ are compatible.
\end{theorem}

\begin{proof}
Let $\beta$ and $\gamma$ be two tagged arcs with underlying
  non-tagged arcs $\beta^\circ$ and $\gamma^\circ$.

  First we study the case where $p=q$. Then the underlying untagged
  arcs are always compatible and hence we only have to care for the
  condition on the tags, which can already be done by looking at the
  case where $p,q\in\{-1,0,\infty\}$. Hence $\alpha_{p,x}$ is
  compatible with $\alpha_{p,y}$ if and only if $x\neq -y$, that is,
  if and only if $\bv_p^x$ is compatible with $\bv_{p}^y$ by Lemma
  \ref{lem:orto}.

  Now, we assume that $p\neq q$ and deal first with the case where
  $p,q\in\{-1,0,\infty\}$. By direct inspection of
Figure~\ref{fig:LowCpx} 
we get that $\alpha_{p,x}$ is compatible with
  $\alpha_{q,y}$ if and only if the condition of the corresponding
  cell in the following table is satisfied.
  \begin{equation}
    \label{eq:comp-spec-tagged-arcs}
    \begin{tabular}{c|c|c|c|}
      & $q=-1$ & $q=0$ & $q=\infty$\\
      \hline
      $p=-1$ & $x\neq -y$ & $x\in H^+ y$ & $y\in H^+ x$ \\
      \hline
      $p=0$ & $y\in H^+ x$  & $x\neq -y$ & $x\in - H^+ y$ \\ 
      \hline
      $p=\infty$& $x\in H^+ y$ & $y\in - H^+ x$ & $x\neq -y$\\
      \hline
    \end{tabular}
  \end{equation}
  Since $\qdist{p,q}\leq 2$ in any such case we get by Proposition
  \ref{prop:orto} that this happens if and only if $\bv_p^x$ is
  compatible with $\bv_q^y$.

  Now, let $p,q\in \QQi$ be arbitrary with $p\neq q$. Notice that
  therefore $\alpha_{p,x}^\circ\neq \alpha_{q,y}^\circ$. Observe also
  that $\typc{p}=\typc{q}$ holds if and only if $\qdist{p,q}$ is even.
  Therefore the two arcs $\alpha_{p,x}$ and $\alpha_{q,y}$ are
  compatible by Proposition \ref{prop:comp-untagged} if and only if
  either $\qdist{p,q}=2$ and the tags on both ends coincide, that is
  $x=y$, or $\qdist{p,q}=1$ and the arcs $\alpha_{\typc{p},x}$ and
  $\alpha_{\typc{q},y}$ are compatible, which happens if and only if
  the condition in \eqref{eq:comp-spec-tagged-arcs} is satisfied. In
  both cases, this is equivalent to the condition that $\bv_p^x$ and
  $\bv_q^y$ are compatible by Proposition~\ref{eq:cond-comp} 
  and Corollary~\ref{cor:orto}.
  \end{proof}

%%%%%%%%%%%%%%%%%%%%%%%%%%%%%%%%%%%%%%%%%%%%%%%%%%
\subsection{Connection to real Schur roots}

We can find by Proposition~\ref{prop:exceptional-to-vectors} and 
Lemma~\ref{lem:correct-slope} for each $(q,x)\in \QQi\times H$ 
up to isomorphism a unique $E^x_q\in\Exc$ such that $\dimv E^x_q =\bv_q^x$.
In fact, we obtain a bijection $(q,x)\mapsto [E^x_q]$ from $\QQi\times H$ 
to the isoclasses of $\Exc$.

\begin{theorem} \label{thm:bijection}
  The map $\Psi:\alpha_{p,x}\mapsto [E^x_p]$ is a bijection between
  the set of all tagged arcs of the $2$-sphere with 4 punctures
  $(\surf,\marked)$ to the set of isoclasses of indecomposable rigid
  objects in
  $\coh \XX$, where $\XX$ has weight type $(2,2,2,2)$.  Moreover,
  two tagged arcs are compatible if and only if their images under 
  $\Psi$ are compatible. Therefore $\Psi$ induces an isomorphism between 
  the exchange graph
  $\tEG(\surf,\marked)$ and the mutation-exchange graph of tilting
  objects in $\coh \XX$. 
\end{theorem}

\begin{proof}
Using the above remark and Definition~\ref{def:tarc} we can define
 a bijection $\Psi'$ from the set of tagged arcs 
$\{\alpha_{p,x}\mid (p,x) \in\QQi\times H\}$
to the  isomorphism classes of indecomposable rigid objects in 
$\coh \XX$ by taking 
$\Psi'(\alpha_{p,x}):=[E^x_p]$ for all $(p,x)\in\QQi\times H$.

By Theorem~\ref{thm:comp} two tagged arcs are
  compatible if and only if their images under $\Psi'$ are Ext-orthogonal 
in $\coh\XX$. It remains to see that each tagged arc of $(\surf,\marked)$ is
  of the form $\alpha_{p,x}$ for some $p\in\QQi$ and some $x\in H$.
  By \cite{FST}, the exchange graph $\tEG(\surf,\marked)$ is connected
  and each tagged triangulation consists of precisely $6$ tagged
  arcs. Consequently if $\beta$ is any tagged arc, we can complete
  $\beta$ to a tagged  triangulation $T$ of $(\surf,\marked)$. Let
  $\mu_{i_1},\ldots,\mu_{i_L}$ be a sequence of mutations such that
  $\mu_{i_L}\cdots\mu_{i_1}(T)=\{\alpha_q^\varepsilon\mid
  q=-1,0,\infty;\ \  \varepsilon=\pm\}=\colon T'$.  Then
  $\mu_{i_1}\cdots\mu_{i_L}(\Psi'(T'))$ is a tilting object in $\coh
  \XX$ and therefore of the form $\Psi'(T'')$ with
  $T''=\mu_{i_1}\cdots\mu_{i_L}(T')=T$ and consequently
  $\beta=\alpha_{p,x}$ for some $p\in\QQ_i$ and $x\in H$.
\end{proof}

%%%%%%%%%%%%%%%%%%%%%%%%%%%%%%%%%%%%%%%%%%%%%%%%%%

\end{document}